\documentclass[12pt]{article}
\usepackage{amsmath,mathdots}
\usepackage{amssymb}
\usepackage{latexsym}
\usepackage{amsthm}
\usepackage{amsmath,amssymb,amsfonts,amsthm}
\usepackage{mdframed}
\usepackage{graphicx}
\usepackage{stmaryrd}
\usepackage{gensymb}
\usepackage{bm}
\usepackage{tikz}  
\usepackage[symbol]{footmisc}
\usepackage{stmaryrd}
\usepackage{multirow}

\usepackage{graphicx}
\usepackage{tikz}
\usetikzlibrary{matrix,arrows}
\usetikzlibrary{positioning}
\usetikzlibrary{fit}
\usetikzlibrary{patterns}

\usepackage[colorlinks,
linkcolor=blue,
anchorcolor=blue,
citecolor=blue
]{hyperref}
\usepackage[margin=2.9cm]{geometry}
\usepackage[normalem]{ulem}

\newtheorem{thm}{Theorem}

\newtheorem{lem}[thm]{Lemma}

\newtheorem{coro}[thm]{Corollary}
\newtheorem{conj}[thm]{Conjecture}

\newtheorem*{remark}{Remark}

\usepackage{graphicx}
\usepackage{tikz}
\usetikzlibrary{matrix,arrows}
\usetikzlibrary{positioning}
\usetikzlibrary{fit}
\usetikzlibrary{patterns}

\newcommand{\T}{\mathcal{T}}
\newcommand{\A}{\mathcal{A}}

\newcommand{\pattern}[4]{										
	\raisebox{0.6ex}{
		\begin{tikzpicture}[scale=0.35, baseline=(current bounding box.center), #1]
		\foreach \x/\y in {#4}		\fill[gray!50] (\x,\y) rectangle +(1,1);
		\draw (0.01,0.01) grid (#2+0.99,#2+0.99);
		\foreach \x/\y in {#3}		\filldraw (\x,\y) circle (6pt);
		\end{tikzpicture}}
}

\newcommand{\dbrac}[1]{{\llbracket #1 \rrbracket}} 	
\newcommand{\boks}[2]{({#1, #2})}   

\definecolor{red}{rgb}{1,0,0}
{}

 \allowdisplaybreaks

\begin{document}

\begin{center}
{\large \bf Joint Equidistributions of Mesh Patterns 123 and 321 with Symmetric and Minus-Antipodal Shadings}
\end{center}

\begin{center}
Shuzhen Lv  and  Philip B. Zhang\footnote[1]{Corresponding author}
\\[6pt]

College of Mathematical Sciences
 \& Institute of Mathematics and Interdisciplinary Sciences
\\
Tianjin Normal University \\ Tianjin  300387, P. R. China\\[6pt]
Email:
      {\tt zhang@tjnu.edu.cn}
\end{center}

\noindent\textbf{Abstract.}
In this paper, we extend recent results by Lv and Kitaev by proving 20~(out of 22 possible) joint equidistributions of mesh patterns 123 and 321 with symmetric shadings, as well as all 36 joint equidistributions of these patterns with minus-antipodal shadings. Our results link several joint equidistributions of mesh patterns to various integer sequences, including unsigned Stirling numbers of the first kind, harmonic numbers, and the numbers of inversion sequences avoiding a certain vincular pattern studied by Lin and Yan. 




\section{Introduction}\label{intro}


The concept of a {\em mesh pattern} was introduced by Br\"{a}nd\'{e}n and Claesson \cite{BrCl}. The study of mesh patterns  has become a prominent area of research within the theory of patterns in permutations, which has attracted significant attention (see~\cite{Kit}). For enumerative results related to mesh patterns, see, for example,~\cite{AKV,  Hilmarsson2015Wilf, KL, KZ, KZZ, KLv, T2}.

An $n$-permutation is an arrangement of $\{1, 2, \ldots, n\}$, and the set of all $n$-permutations is denoted by $S_n$. 
 A mesh pattern of length $k$ is a pair $(\tau,R)$, where $\tau \in S_k$ and~$R$ is a subset of $\dbrac{0,k} \times \dbrac{0,k}$, where
$\dbrac{0,k}$ denotes the interval of integers from $0$ to $k$.
Let $\boks{i}{j}$ be the box whose corners have coordinates $(i,j), (i,j+1),
(i+1,j+1)$, and $(i+1,j)$, and let the horizontal lines (resp.,   vertical lines) represent the values (resp.,   positions) in the pattern. Mesh patterns can be drawn by shading the boxes in $R$. A mesh pattern with $\tau=123$ and $R = \{\boks{0}{0},\boks{1}{2},\boks{2}{1},\boks{3}{1}\}$ is drawn as follows
\[
\pattern{scale=0.8}{3}{1/1,2/2,3/3}{0/0,1/2, 2/1,3/1}.
\] 
A mesh pattern has a {\em symmetric shading}
if the shading of the box $(i,j)$  implies the shading of  the box $(j,i)$, and 
 {\em minus-antipodal shading} if, for any $i\neq j$, precisely  one of the boxes $(i,j)$ and $(j,i)$ is shaded. See Table~\ref{tab-sy-shadings} (resp., Table~\ref{tab-anti-shadings}) for examples of symmetric (resp., minus-antipodal) shadings of mesh patterns. 
 
Given a permutation $\pi = \pi_1 \pi_2 \cdots \pi_n \in S_n$, a subsequence $\pi' = \pi_{i_1} \pi_{i_2} \cdots \pi_{i_k}$ is an \emph{occurrence} of a mesh pattern $p = (\tau, R)$ in $\pi$ if $(i)$ $\pi'$ is order-isomorphic to $\tau$, and $(ii)$ the shaded regions defined by $R$ contain no elements of $\pi$ that are not in $\pi'$.
Otherwise, $\pi$ is said to avoid $p$, denoted by $\pi \in S_n(p)$. 
 For example, consider the mesh pattern $p=\bigl(123,\{(0,0),(1,2),(2,1),(3,1)\}\bigr)$.  
The permutation $23154$ contains two occurrences of $p$, namely $235$ and $234$.
However, the permutation $14325$ contains three occurrences of the pattern $123$, namely $145$, $135$ and $125$, but it avoids $p$ because some entries lie inside the shaded regions, as shown in Fig~\ref{fig-exa-14325}.

\begin{figure}[h]
\begin{center}
\setlength{\tabcolsep}{4pt}   
\begin{tabular}{ccc}
\begin{tikzpicture}[scale=0.4]
  \tikzset{grid/.style={draw,step=1cm,gray!100,thin},
           graycell/.style={fill=gray!50,draw=none,minimum width=1cm,minimum height=1cm,anchor=south west}}
  \fill[graycell] (0,3) rectangle (1,4);
  \fill[graycell] (1,3) rectangle (4,0);
  \draw[grid] (0,0) grid (4,4);
  
  \draw[thick] (0,0) circle (5pt); 
  \draw[thick] (0,0) circle (2pt);
  \draw[thick] (1,3) circle (5pt);
  \draw[thick] (1,3) circle (2pt);
  \filldraw[black] (2,2) circle (3pt);
  \filldraw[black] (3,1) circle (3pt);
  \draw[thick] (4,4) circle (5pt);
  \draw[thick] (4,4) circle (2pt);
\end{tikzpicture}
&
\begin{tikzpicture}[scale=0.4]
  \tikzset{grid/.style={draw,step=1cm,gray!100,thin},
           graycell/.style={fill=gray!50,draw=none,minimum width=1cm,minimum height=1cm,anchor=south west}}
  \fill[graycell] (0,2) rectangle (2,4);
  \fill[graycell] (2,2) rectangle (4,0);
  \draw[grid] (0,0) grid (4,4);
  
  \draw[thick] (0,0) circle (5pt); 
  \draw[thick] (0,0) circle (2pt);
  \filldraw[black] (1,3) circle (3pt);
  \draw[thick] (2,2) circle (5pt); 
  \draw[thick] (2,2) circle (2pt);
  \filldraw[black] (3,1) circle (3pt);
  \draw[thick] (4,4) circle (5pt); 
  \draw[thick] (4,4) circle (2pt);
\end{tikzpicture}
&
\begin{tikzpicture}[scale=0.4]
  \tikzset{grid/.style={draw,step=1cm,gray!100,thin},
           graycell/.style={fill=gray!50,draw=none,minimum width=1cm,minimum height=1cm,anchor=south west}}
  \fill[graycell] (0,1) rectangle (3,4);
  \fill[graycell] (3,1) rectangle (4,0);
  \draw[grid] (0,0) grid (4,4);
  
  \draw[thick] (0,0) circle (5pt); 
  \draw[thick] (0,0) circle (2pt);
  \filldraw[black] (1,3) circle (3pt);
  \filldraw[black] (2,2) circle (3pt);
  \draw[thick] (3,1) circle (5pt); 
  \draw[thick] (3,1) circle (2pt);
  \draw[thick] (4,4) circle (5pt); 
  \draw[thick] (4,4) circle (2pt);
\end{tikzpicture}
\end{tabular}
\end{center}
\caption{An example of the permutation $14325$ avoiding $p$, where the circled dots in the diagrams represent the respective occurrences of the pattern $123$.}\label{fig-exa-14325}
\end{figure}

Two patterns $q_1$ and $q_2$ are {\em Wilf-equivalent}, denoted $q_1\sim q_2$, if  $|S_n(q_1)|=|S_n(q_2)|$. Moreover, $q_1$ and $q_2$ are called {\em equidistributed}, denoted~$q_1\sim_d q_2$, 
if the number of permutations of length~$n$ with $k$ occurrences of $q_1$ is equal to that with $k$ occurrences of $q_2$ for any nonnegative integers $n$ and $k$. The stronger notion of \textit{joint equidistribution} of two patterns, denoted $q_1\sim_{jd}q_2$, means that the number of $n$-permutations with $k$ occurrences of $q_1$ and $\ell$ occurrences of $q_2$ is equal to that with $\ell$  occurrences of $q_1$ and $k$ occurrences of $q_2$. 

Hilmarsson et al.~\cite{Hilmarsson2015Wilf} initiated the systematic study of the avoidance of mesh patterns. This research was later extended by Kitaev and Zhang in~\cite{KZ}, focusing on the distribution of mesh patterns. Subsequently, Han and Zeng~\cite{Han-Zeng} resolved several distribution conjectures posed by Kitaev and Zhang. More recently, Lv and Kitaev~\cite{KLv} established 75 joint equidistribution results for the mesh patterns 123 and 132 with identical symmetric shadings, primarily using bijections.

Note that the permutations 123 and 321 remain invariant under the group-theoretic inverse, which corresponds to the reflection of the elements across the diagonal in the respective permutation matrices. This observation motivates our study of the joint equidistributions of mesh patterns 123 and 321 with certain shadings, extending the results of Lv and Kitaev. 

In this paper, as suggested by computer experiments, we derive 56 of the 58 possible joint equidistributions of the mesh patterns 123 and 321 as presented in Tables~\ref{tab-sy-shadings} and~\ref{tab-anti-shadings}, which arise from 20 identical symmetric shadings  and 36  minus-antipodal shadings  considered  in Sects.~\ref{sec-sy} and~\ref{sec-an-sy}, respectively.  We also propose the remaining two pairs as conjectures.
Our main results are obtained through recurrence relations for the joint distribution generating functions, which lead to expressions that connect to several integer sequences. 
The most interesting pairs, Nr.~A17--A24 discussed in Theorem~\ref{thm-pairs-39-47}, have joint distributions involving both the \emph{unsigned Stirling numbers of the first kind} and the \emph{harmonic numbers}.
It is also worth pointing out that one of the patterns in any pair among Nr.~A25--A36 is equidistributed, as deduced in Corollary~\ref{coro-equi}, and is linked to the \emph{inversion sequences} avoiding a certain \emph{vincular pattern}.

This paper is organized as follows. In Sect. ~\ref{sec-sy}, we consider 20 joint equidistributions with symmetric shadings, summarizing our results in Table~\ref{tab-sy-shadings}. The joint equidistribution of two pairs is left as a conjecture (see Conjecture~\ref{con-pairs-21-22}), which we confirm in the case of avoidance by proving the Wilf-equivalence of the pairs S21 and S22 (see Theorem~\ref{thm-pairs-21-22}).

In Sect.~\ref{sec-an-sy}, we prove 36  joint equidistributions with minus-antipodal shadings and present our results in Table~\ref{tab-anti-shadings}. These results are connected to many known integer sequences, both with and without combinatorial interpretations, which naturally arise within our framework. In particular, the formula in~\eqref{T-pair-39-generating} is related to the unsigned Stirling number of the first kind, while the formula in~\eqref{eq-har} corresponds to the harmonic numbers. Additionally, the formula in~\eqref{I_n,k=T_n,k} is related to the number of inversion sequences avoiding a specific vincular pattern, as obtained by Lin and Yan~\cite{Lin-Yan}.

\begin{table}
 	{
 		\renewcommand{\arraystretch}{1.6}
 		\setlength{\tabcolsep}{5pt}
 \begin{center} 
 		\begin{tabular}{|c c c c c c c
 		c|c|}
 			\hline
 		\footnotesize{Nr.}	& {\footnotesize Patterns}  & \footnotesize{Nr.}	& {\footnotesize Patterns}  & \footnotesize{Nr.}	& {\footnotesize Patterns}  & \footnotesize{Nr.}	& {\footnotesize Patterns} & {\footnotesize Reference}  \\
 			\hline
 			\hline

\footnotesize{S1} & \multicolumn{1}{c|}{$\pattern{scale = 0.5}{3}{1/1,2/2,3/3}{0/0,0/1,0/2,0/3,1/0,1/1,1/2,1/3,2/0,2/1,2/2,2/3,3/0,3/1,3/2,3/3}\pattern{scale = 0.5}{3}{1/3,2/2,3/1}{0/0,0/1,0/2,0/3,1/0,1/1,1/2,1/3,2/0,2/1,2/2,2/3,3/0,3/1,3/2,3/3} $} & 
\footnotesize{S2} & \multicolumn{1}{c|}{$\pattern{scale = 0.5}{3}{1/1,2/2,3/3}{0/0,0/1,0/2,0/3,1/0,1/3,2/0,2/3,3/0,3/1,3/2,3/3}\pattern{scale = 0.5}{3}{1/3,2/2,3/1}{0/0,0/1,0/2,0/3,1/0,1/3,2/0,2/3,3/0,3/1,3/2,3/3}$}   &
\footnotesize{S3} & \multicolumn{1}{c|}{$\pattern{scale = 0.5}{3}{1/1,2/2,3/3}{0/0,0/3,1/1,1/2,2/1,2/2,3/0,3/3}\pattern{scale = 0.5}{3}{1/3,2/2,3/1}{0/0,0/3,1/1,1/2,2/1,2/2,3/0,3/3}$} & 
\footnotesize{S4} & $\pattern{scale = 0.5}{3}{1/1,2/2,3/3}{0/0,0/3,3/0,3/3}\pattern{scale = 0.5}{3}{1/3,2/2,3/1}{0/0,0/3,3/0,3/3}$ &\footnotesize{\multirow{2.5}{2.2cm}{Subsection~\ref{sub-sy-cr}}}

\\[5pt]
\cline{1-8}
\footnotesize{S5} & \multicolumn{1}{c|}{$\pattern{scale = 0.5}{3}{1/1,2/2,3/3}{0/1,0/2,1/0,1/1,1/2,1/3,2/0,2/1,2/2,2/3,3/1,3/2}\pattern{scale = 0.5}{3}{1/3,2/2,3/1}{0/1,0/2,1/0,1/1,1/2,1/3,2/0,2/1,2/2,2/3,3/1,3/2}$} &
\footnotesize{S6} & \multicolumn{1}{c|}{$\pattern{scale = 0.5}{3}{1/1,2/2,3/3}{0/1,0/2,1/0,1/3,2/0,2/3,3/1,3/2}\pattern{scale = 0.5}{3}{1/3
,2/2,3/1}{0/1,0/2,1/0,1/3,2/0,2/3,3/1,3/2}$} &
\footnotesize{S7} & \multicolumn{1}{c|}{$\pattern{scale = 0.5}{3}{1/1,2/2,3/3}{1/1,1/2,2/1,2/2}\pattern{scale = 0.5}{3}{1/3,2/2,3/1}{1/1,1/2,2/1,2/2}$} & 
\footnotesize{S8} &  \multicolumn{1}{c|}{$\pattern{scale = 0.5}{3}{1/1,2/2,3/3}{}\pattern{scale = 0.5}{3}{1/3,2/2,3/1}{}$}&
\\[5pt]
\hline
\footnotesize{S9} & $\pattern{scale = 0.5}{3}{1/1,2/2,3/3}{0/0,0/1,0/2,0/3,1/0,1/1,1/2,1/3,2/0,2/1,2/2,2/3,3/0,3/1,3/2}\pattern{scale = 0.5}{3}{1/3,2/2,3/1}{0/0,0/1,0/2,0/3,1/0,1/1,1/2,1/3,2/0,2/1,2/2,2/3,3/0,3/1,3/2}$ & 
\footnotesize{S10} & \multicolumn{1}{c|}{$\pattern{scale = 0.5}{3}{1/1,2/2,3/3}{0/1,0/2,0/3,3/0,1/0,1/1,1/2,1/3,2/0,2/1,2/2,2/3,3/1,3/2,3/3}\pattern{scale = 0.5}{3}{1/3,2/2,3/1}{0/1,0/2,1/0,0/3,3/0,1/1,1/2,1/3,2/0,2/1,2/2,2/3,3/1,3/2,3/3}$} &
\footnotesize{S11} & $\pattern{scale = 0.5}{3}{1/1,2/2,3/3}{0/0,0/1,0/2,0/3,1/0,1/1,1/2,1/3,2/0,2/1,2/3,3/0,3/1,3/2,3/3}\pattern{scale = 0.5}{3}{1/3,2/2,3/1} {0/0,0/1,0/2,0/3,1/0,1/1,1/2,1/3,2/0,2/1,2/3,3/0,3/1,3/2,3/3}$ &
\footnotesize{S12} & \multicolumn{1}{c|}{$\pattern{scale = 0.5}{3}{1/1,2/2,3/3}{0/0,0/1,0/2,0/3,1/0,1/2,1/3,2/0,2/1,2/2,2/3,3/0,3/1,3/2,3/3}\pattern{scale = 0.5}{3}{1/3,2/2,3/1}{0/0,0/1,0/2,0/3,1/0,1/2,1/3,2/0,2/1,2/2,2/3,3/0,3/1,3/2,3/3}$} &\footnotesize{\multirow{3.8}{2.2cm}{Subsection~\ref{sub-dir-bijections}}}
\\[5pt]
\cline{1-8}
 \footnotesize{S13} & $\pattern{scale = 0.5}{3}{1/1,2/2,3/3}{0/0,0/1,0/2,0/3,1/0,1/1,1/3,2/0,2/2,2/3,3/0,3/1,3/2,3/3}\pattern{scale = 0.5}{3}{1/3,2/2,3/1}{0/0,0/1,0/2,0/3,1/0,1/1,1/3,2/0,2/2,2/3,3/0,3/1,3/2,3/3}$ & 
\footnotesize{S14} & \multicolumn{1}{c|}{$\pattern{scale = 0.5}{3}{1/1,2/2,3/3}{0/0,0/1,0/2,0/3,1/0,1/2,1/3,2/0,2/1,2/3,3/0,3/1,3/2,3/3}\pattern{scale = 0.5}{3}{1/3,2/2,3/1}{0/0,0/1,0/2,0/3,1/0,1/2,1/3,2/0,2/1,2/3,3/0,3/1,3/2,3/3}$}  &
\footnotesize{S15} & $\pattern{scale = 0.5}{3}{1/1,2/2,3/3}{0/0,0/1,0/2,0/3,1/0,1/1,1/3,2/0,2/3,3/0,3/1,3/2,3/3}\pattern{scale = 0.5}{3}{1/3,2/2,3/1} {0/0,0/1,0/2,0/3,1/0,1/1,1/3,2/0,2/3,3/0,3/1,3/2,3/3}$ &
\footnotesize{S16} & \multicolumn{1}{c|}{$\pattern{scale = 0.5}{3}{1/1,2/2,3/3}{0/0,0/1,0/2,0/3,1/0,1/3,2/0,2/2,2/3,3/0,3/1,3/2,3/3}\pattern{scale = 0.5}{3}{1/3,2/2,3/1}{0/0,0/1,0/2,0/3,1/0,1/3,2/0,2/2,2/3,3/0,3/1,3/2,3/3}$} & \\[5pt]
\cline{1-8}
\footnotesize{S17} & $\pattern{scale = 0.5}{3}{1/1,2/2,3/3}{0/0,0/1,0/2,0/3,1/0,1/1,1/3,2/0,2/2,2/3,3/0,3/1,3/2}\pattern{scale = 0.5}{3}{1/3,2/2,3/1} {0/0,0/1,0/2,0/3,1/0,1/1,1/3,2/0,2/2,2/3,3/0,3/1,3/2}$ &
\footnotesize{S18} & $\pattern{scale = 0.5}{3}{1/1,2/2,3/3}{0/1,0/2,0/3,1/0,1/1,1/3,2/0,2/2,2/3,3/0,3/1,3/2,3/3}\pattern{scale = 0.5}{3}{1/3,2/2,3/1}{0/1,0/2,0/3,1/0,1/1,1/3,2/0,2/2,2/3,3/0,3/1,3/2,3/3}$ &
& & & &\\[5pt]
\hline
 \footnotesize{S19} & $\pattern{scale = 0.5}{3}{1/1,2/2,3/3}{0/0,0/1,0/2,1/0,1/1,1/2,2/0,2/1,2/2}\pattern{scale = 0.5}{3}{1/3,2/2,3/1}{0/0,0/1,0/2,1/0,1/1,1/2,2/0,2/1,2/2}$  &
\footnotesize{S20} &
$\pattern{scale = 0.5}{3}{1/1,2/2,3/3}{1/1,1/2,1/3,2/1,2/2,2/3,3/1,3/2,3/3}\pattern{scale = 0.5}{3}{1/3,2/2,3/1}{1/1,1/2,1/3,2/1,2/2,2/3,3/1,3/2,3/3}$   & & & &  &\footnotesize{\multirow{1.2}{1.9cm}{Theorem~\ref{thm-pairs-19-20}}}
\\[5pt]
\hline
\footnotesize{S21} & $\pattern{scale = 0.5}{3}{1/1,2/2,3/3}{0/0,0/1,0/2,1/0,1/1,1/2,2/0,2/1,2/2,3/3}\pattern{scale = 0.5}{3}{1/3,2/2,3/1}{0/0,0/1,0/2,1/0,1/1,1/2,2/0,2/1,2/2,3/3}$  &
\footnotesize{S22} &
$\pattern{scale = 0.5}{3}{1/1,2/2,3/3}{0/0,1/1,1/2,1/3,2/1,2/2,2/3,3/1,3/2,3/3}\pattern{scale = 0.5}{3}{1/3,2/2,3/1}{0/0,1/1,1/2,1/3,2/1,2/2,2/3,3/1,3/2,3/3}$&  & & &  &\footnotesize{\multirow{1.2}{2.2cm}{Conjecture~\ref{con-pairs-21-22}}}
\\[8pt]
\hline
	\end{tabular}
\end{center} 
}
\vspace{-0.5cm}
 	\caption{Joint equidistributions for symmetric shadings. Pairs sharing the same joint equidistribution are grouped within the same frame.}\label{tab-sy-shadings}
\end{table}

\begin{table}[ht]
 	{
 		\renewcommand{\arraystretch}{1.5}
 		\setlength{\tabcolsep}{6pt}
 \begin{center} 
 		\begin{tabular}{|c c c c  c c c
 		c|c|}
 			\hline
 		\footnotesize{Nr.}	& {\footnotesize Patterns}  & \footnotesize{Nr.}	& {\footnotesize Patterns}  & \footnotesize{Nr.}	& {\footnotesize Patterns}  & \footnotesize{Nr.}	& {\footnotesize Patterns} & {\footnotesize Reference} \\
 			\hline
 			\hline

\footnotesize{A1} & $\pattern{scale = 0.5}{3}{1/1,2/2,3/3}{0/0,0/1,0/2,0/3,1/1,1/2,3/1,3/2}\pattern{scale = 0.5}{3}{1/3,2/2,3/1}{0/0,0/1,0/2,0/3,1/1,1/2,3/1,3/2}$  &
\footnotesize{A2} & $\pattern{scale = 0.5}{3}{1/1,2/2,3/3}{0/0,1/0,2/0,3/0,1/1,2/1,1/3,2/3}\pattern{scale = 0.5}{3}{1/3,2/2,3/1}{0/0,1/0,2/0,3/0,1/1,2/1,1/3,2/3}$   &
\footnotesize{A3} & $\pattern{scale = 0.5}{3}{1/1,2/2,3/3}{0/1,0/2,2/1,2/2,3/0,3/1,3/2,3/3}\pattern{scale = 0.5}{3}{1/3,2/2,3/1}{0/1,0/2,2/1,2/2,3/0,3/1,3/2,3/3}$  & 
\footnotesize{A4} & $\pattern{scale = 0.5}{3}{1/1,2/2,3/3}{1/0,1/2,2/0,2/2,0/3,1/3,2/3,3/3}\pattern{scale = 0.5}{3}{1/3
,2/2,3/1}{1/0,1/2,2/0,2/2,0/3,1/3,2/3,3/3}$  &\footnotesize{\multirow{5}{2.2cm}{Subsection~\ref{subsec-an-cr}}} \\[5pt]
\cline{1-8}
\footnotesize{A5} & $\pattern{scale = 0.5}{3}{1/1,2/2,3/3}{0/0,0/1,0/2,0/3,2/1,2/2,3/1,3/2}\pattern{scale = 0.5}{3}{1/3,2/2,3/1}{0/0,0/1,0/2,0/3,2/1,2/2,3/1,3/2}$   &
\footnotesize{A6} & $\pattern{scale = 0.5}{3}{1/1,2/2,3/3}{0/0,1/0,2/0,3/0,1/2,2/2,1/3,2/3}\pattern{scale = 0.5}{3}{1/3,2/2,3/1}{0/0,1/0,2/0,3/0,1/2,2/2,1/3,2/3}$  &
\footnotesize{A7} & $\pattern{scale = 0.5}{3}{1/1,2/2,3/3}{0/1,0/2,1/1,1/2,3/0,3/1,3/2,3/3}\pattern{scale = 0.5}{3}{1/3,2/2,3/1}{0/1,0/2,1/1,1/2,3/0,3/1,3/2,3/3}$ & 
\footnotesize{A8} & $\pattern{scale = 0.5}{3}{1/1,2/2,3/3}{1/0,1/1,2/0,2/1,0/3,1/3,2/3,3/3}\pattern{scale = 0.5}{3}{1/3,2/2,3/1}{1/0,1/1,2/0,2/1,0/3,1/3,2/3,3/3}$ &
\\[5pt]
\cline{1-8}
\footnotesize{A9} & $\pattern{scale = 0.5}{3}{1/1,2/2,3/3}{0/1,0/2,0/3,1/1,2/1,3/1,3/2,3/3}\pattern{scale = 0.5}{3}{1/3,2/2,3/1}{0/1,0/2,0/3,1/1,2/1,3/1,3/2,3/3}$ &
\footnotesize{A10} & $\pattern{scale = 0.5}{3}{1/1,2/2,3/3}{0/0,0/1,0/2,1/2,2/2,3/0,3/1,3/2}\pattern{scale = 0.5}{3}{1/3,2/2,3/1}{0/0,0/1,0/2,1/2,2/2,3/0,3/1,3/2}$  &
\footnotesize{A11} &$\pattern{scale = 0.5}{3}{1/1,2/2,3/3}{0/0,0/3,1/0,1/3,2/0,2/1,2/2,2/3}\pattern{scale = 0.5}{3}{1/3
,2/2,3/1}{0/0,0/3,1/0,1/3,2/0,2/1,2/2,2/3}$  & 
\footnotesize{A12} & $\pattern{scale = 0.5}{3}{1/1,2/2,3/3}{1/0,1/1,1/2,1/3,2/0,2/3,3/0,3/3}\pattern{scale = 0.5}{3}{1/3,2/2,3/1}{1/0,1/1,1/2,1/3,2/0,2/3,3/0,3/3}$ &
\\[5pt]
\cline{1-8}
\footnotesize{A13} & $\pattern{scale = 0.5}{3}{1/1,2/2,3/3}{0/0,0/1,0/2,1/1,2/1,3/0,3/1,3/2}\pattern{scale = 0.5}{3}{1/3,2/2,3/1}{0/0,0/1,0/2,1/1,2/1,3/0,3/1,3/2}$ &
\footnotesize{A14} & $\pattern{scale = 0.5}{3}{1/1,2/2,3/3}{0/0,0/3,1/0,1/1,1/2,1/3,2/0,2/3}\pattern{scale = 0.5}{3}{1/3,2/2,3/1}{0/0,0/3,1/0,1/1,1/2,1/3,2/0,2/3}$ &
\footnotesize{A15} & $\pattern{scale = 0.5}{3}{1/1,2/2,3/3}{0/1,0/2,0/3,1/2,2/2,3/1,3/2,3/3}\pattern{scale = 0.5}{3}{1/3
,2/2,3/1}{0/1,0/2,0/3,1/2,2/2,3/1,3/2,3/3}$ & 
\footnotesize{A16} & $\pattern{scale = 0.5}{3}{1/1,2/2,3/3}{1/0,2/0,3/0,2/1,2/2,1/3,2/3,3/3}\pattern{scale = 0.5}{3}{1/3,2/2,3/1}{1/0,2/0,3/0,2/1,2/2,1/3,2/3,3/3}$ &
\\[5pt]
			\hline
\footnotesize{A17} & $\pattern{scale = 0.5}{3}{1/1,2/2,3/3}{0/0,0/1,0/2,0/3,1/1,2/1,3/1,3/2}\pattern{scale = 0.5}{3}{1/3,2/2,3/1}{0/0,0/1,0/2,0/3,1/1,2/1,3/1,3/2}$  &
\footnotesize{A18} &  $\pattern{scale = 0.5}{3}{1/1,2/2,3/3}{0/0,0/1,0/2,0/3,1/2,2/2,3/2,3/1}\pattern{scale = 0.5}{3}{1/3,2/2,3/1}{0/0,0/1,0/2,0/3,1/2,2/2,3/2,3/1}$    &
\footnotesize{A19} & $\pattern{scale = 0.5}{3}{1/1,2/2,3/3}{0/0,1/0,2/0,3/0,1/1,1/2,1/3,2/3}\pattern{scale = 0.5}{3}{1/3,2/2,3/1}{0/0,1/0,2/0,3/0,1/1,1/2,1/3,2/3}$ & 
\footnotesize{A20} & $\pattern{scale = 0.5}{3}{1/1,2/2,3/3}{0/0,1/0,2/0,3/0,2/1,2/2,2/3,1/3}\pattern{scale = 0.5}{3}{1/3,2/2,3/1}{0/0,1/0,2/0,3/0,2/1,2/2,2/3,1/3}$ &\footnotesize{\multirow{2.5}{1.9cm}{Theorem~\ref{thm-pairs-39-47}}} \\[5pt]

\footnotesize{A21} & $\pattern{scale = 0.5}{3}{1/1,2/2,3/3}{0/1,0/2,1/1,2/1,3/1,3/0,3/2,3/3}\pattern{scale = 0.5}{3}{1/3,2/2,3/1}{0/1,0/2,1/1,2/1,3/1,3/0,3/2,3/3}$ &
\footnotesize{A22} & $\pattern{scale = 0.5}{3}{1/1,2/2,3/3}{0/1,0/2,1/2,2/2,3/2,3/1,3/0,3/2,3/3}\pattern{scale = 0.5}{3}{1/3
,2/2,3/1}{0/1,0/2,1/2,2/2,3/2,3/0,3/1,3/2,3/3}$ &
\footnotesize{A23} & $\pattern{scale = 0.5}{3}{1/1,2/2,3/3}{0/3,1/3,2/3,3/3,1/0,2/0,2/1,2/2}\pattern{scale = 0.5}{3}{1/3,2/2,3/1}{0/3,1/3,2/3,3/3,1/0,2/0,2/1,2/2}$ &
\footnotesize{A24} & $\pattern{scale = 0.5}{3}{1/1,2/2,3/3}{0/3,1/3,2/3,3/3,1/0,2/0,1/1,1/2}\pattern{scale = 0.5}{3}{1/3
,2/2,3/1}{0/3,1/3,2/3,3/3,1/0,2/0,1/1,1/2}$ &
\\[5pt]
\hline	
\footnotesize{A25} &  $\pattern{scale = 0.5}{3}{1/1,2/2,3/3}{0/0,0/1,0/2,1/1,1/2,3/0,3/1,3/2}\pattern{scale = 0.5}{3}{1/3,2/2,3/1}{0/0,0/1,0/2,1/1,1/2,3/0,3/1,3/2}$ &
\footnotesize{A26} & $\pattern{scale = 0.5}{3}{1/1,2/2,3/3}{1/0,2/0,3/0,1/1,2/1,1/3,2/3,3/3}\pattern{scale = 0.5}{3}{1/3
,2/2,3/1}{1/0,2/0,3/0,1/1,2/1,1/3,2/3,3/3}$  &
\footnotesize{A27} & $\pattern{scale = 0.5}{3}{1/1,2/2,3/3}{0/1,0/2,0/3,2/1,2/2,3/1,3/2,3/3}\pattern{scale = 0.5}{3}{1/3,2/2,3/1}{0/1,0/2,0/3,2/1,2/2,3/1,3/2,3/3}$ & 
\footnotesize{A28} & $\pattern{scale = 0.5}{3}{1/1,2/2,3/3}{0/1,0/2,0/3,1/1,1/2,3/1,3/2,3/3}\pattern{scale = 0.5}{3}{1/3,2/2,3/1}{0/1,0/2,0/3,1/1,1/2,3/1,3/2,3/3}$ & \footnotesize{\multirow{2.5}{1.9cm}{Theorem~\ref{thm-pairs-A25-32}}}\\[5pt]

\footnotesize{A29} & $\pattern{scale = 0.5}{3}{1/1,2/2,3/3}{0/0,1/0,2/0,1/2,2/2,0/3,1/3,2/3}\pattern{scale = 0.5}{3}{1/3,2/2,3/1}{0/0,1/0,2/0,1/2,2/2,0/3,1/3,2/3}$  &
\footnotesize{A30} & $\pattern{scale = 0.5}{3}{1/1,2/2,3/3}{0/0,1/0,2/0,1/1,2/1,0/3,1/3,2/3}\pattern{scale = 0.5}{3}{1/3,2/2,3/1}{0/0,1/0,2/0,1/1,2/1,0/3,1/3,2/3}$ &
\footnotesize{A31} & $\pattern{scale = 0.5}{3}{1/1,2/2,3/3}{1/0,2/0,3/0,1/2,2/2,1/3,2/3,3/3}\pattern{scale = 0.5}{3}{1/3,2/2,3/1}{1/0,2/0,3/0,1/2,2/2,1/3,2/3,3/3}$  & 
\footnotesize{A32} & $\pattern{scale = 0.5}{3}{1/1,2/2,3/3}{0/0,0/1,0/2,2/1,2/2,3/0,3/1,3/2}\pattern{scale = 0.5}{3}{1/3,2/2,3/1}{0/0,0/1,0/2,2/1,2/2,3/0,3/1,3/2}$ &
\\[5pt]
\hline
\footnotesize{A33} & $\pattern{scale = 0.5}{3}{1/1,2/2,3/3}{0/2,1/0,1/1,1/2,2/2,3/0,3/1,3/2}\pattern{scale = 0.5}{3}{1/3,2/2,3/1}{0/2,1/0,1/1,1/2,2/2,3/0,3/1,3/2}$ &
\footnotesize{A34} & $\pattern{scale = 0.5}{3}{1/1,2/2,3/3}{0/1,0/2,0/3,1/1,2/1,2/2,2/3,3/1}\pattern{scale = 0.5}{3}{1/3,2/2,3/1}{0/1,0/2,0/3,1/1,2/1,2/2,2/3,3/1}$ &
\footnotesize{A35} & $\pattern{scale = 0.5}{3}{1/1,2/2,3/3}{0/1,0/3,1/1,1/3,2/0,2/1,2/1,2/2,2/3}\pattern{scale = 0.5}{3}{1/3,2/2,3/1}{0/1,0/3,1/1,1/3,2/0,2/1,2/1,2/2,2/3}$  & 
\footnotesize{A36} & $\pattern{scale = 0.5}{3}{1/1,2/2,3/3}{1/0,1/1,1/2,1/3,2/0,2/2,3/0,3/2}\pattern{scale = 0.5}{3}{1/3,2/2,3/1}{1/0,1/1,1/2,1/3,2/0,2/2,3/0,3/2}$ & \footnotesize{\multirow{1.2}{1.9cm}{Theorem~\ref{thm-pair-A33-36}}}\\[5pt]
\hline
	\end{tabular}
\end{center} 
}
\vspace{-0.5cm}
 	\caption{Joint equidistributions for minus-antipodal shadings. Pairs sharing the same joint equidistribution are grouped within the same frame}\label{tab-anti-shadings}
\end{table}

Throughout this paper, let $q_1$ and $q_2$ denote the mesh patterns $123$ and $321$ in question, respectively, with identical shadings. Also, let $\mathcal{T}(n,k,\ell)$ be the set of all $n$-permutations containing exactly $k$ occurrences of $q_1$ and $\ell$ occurrences of $q_2$, and let 
\[T_{n,k,\ell} := |\mathcal{T}(n,k,\ell)|.
\]
Finally, denote the bivariate generating function
\[
T_n(x,y) = \sum_{k=0}^{n-2} \sum_{\ell=0}^{n-2} T_{n,k,\ell} \, x^k y^\ell,
\]
which enumerates the joint distribution of both pattern occurrences in $S_n$. 

\section{Symmetric shadings}\label{sec-sy}

In this section, we consider 20 jointly equidistributed pairs and two Wilf-equivalent pairs of patterns with symmetric shadings, as shown in Table~\ref{tab-sy-shadings}. In Sects.~\ref{sub-sy-cr} and~\ref{sub-dir-bijections}, we obtain the joint equidistribution results by constructing bijections. 
For each pair of patterns, we use $f$ to denote the bijection that swaps the occurrences of the patterns in question. 
In Sect.~\ref{sub-recur-relat},  we get two jointly equidistributed results via recurrence relations. 
In Sect.\ref{sub-Wilf}, we conjecture two pairs of joint equidistribution and prove their Wilf-equivalence by establishing a bijection. 

\subsection{Results obtained by complement or reverse}\label{sub-sy-cr}

In this subsection, we consider the following eight pairs:\vspace{0.1cm}
\begin{center}
\begin{tabular}{rrrrr}
S1 = $\pattern{scale = 0.4}{3}{1/1,2/2,3/3}{0/0,0/1,0/2,0/3,1/0,1/1,1/2,1/3,2/0,2/1,2/2,2/3,3/0,3/1,3/2,3/3}\pattern{scale = 0.4}{3}{1/3,2/2,3/1}{0/0,0/1,0/2,0/3,1/0,1/1,1/2,1/3,2/0,2/1,2/2,2/3,3/0,3/1,3/2,3/3} $; & \hspace{-0.3cm}
S2 = $\pattern{scale = 0.4}{3}{1/1,2/2,3/3}{0/0,0/1,0/2,0/3,1/0,1/3,2/0,2/3,3/0,3/1,3/2,3/3}\pattern{scale = 0.4}{3}{1/3,2/2,3/1}{0/0,0/1,0/2,0/3,1/0,1/3,2/0,2/3,3/0,3/1,3/2,3/3}$;   &\hspace{-0.3cm}
S3 = $\pattern{scale = 0.4}{3}{1/1,2/2,3/3}{0/0,0/3,1/1,1/2,2/1,2/2,3/0,3/3}\pattern{scale = 0.4}{3}{1/3,2/2,3/1}{0/0,0/3,1/1,1/2,2/1,2/2,3/0,3/3}$; & \hspace{-0.3cm}
S4 = $\pattern{scale = 0.4}{3}{1/1,2/2,3/3}{0/0,0/3,3/0,3/3}\pattern{scale = 0.4}{3}{1/3,2/2,3/1}{0/0,0/3,3/0,3/3}$; \\[5pt]
S5 = $\pattern{scale = 0.4}{3}{1/1,2/2,3/3}{0/1,0/2,1/0,1/1,1/2,1/3,2/0,2/1,2/2,2/3,3/1,3/2}\pattern{scale = 0.4}{3}{1/3,2/2,3/1}{0/1,0/2,1/0,1/1,1/2,1/3,2/0,2/1,2/2,2/3,3/1,3/2}$; &\hspace{-0.3cm}
S6 = $\pattern{scale = 0.4}{3}{1/1,2/2,3/3}{0/1,0/2,1/0,1/3,2/0,2/3,3/1,3/2}\pattern{scale = 0.4}{3}{1/3
,2/2,3/1}{0/1,0/2,1/0,1/3,2/0,2/3,3/1,3/2}$; &\hspace{-0.3cm}
S7 = $\pattern{scale = 0.4}{3}{1/1,2/2,3/3}{1/1,1/2,2/1,2/2}\pattern{scale = 0.4}{3}{1/3,2/2,3/1}{1/1,1/2,2/1,2/2}$;& \hspace{-0.3cm}
S8 = $\pattern{scale = 0.4}{3}{1/1,2/2,3/3}{}\pattern{scale = 0.4}{3}{1/3,2/2,3/1}{}$.
\end{tabular}
\end{center} 

Recall that the {\em complement, reverve, and inverse} of $\pi=\pi_1\pi_2\cdots \pi_n \in S_n$, are respectively defined as $\pi^c= (n + 1 - \pi_1)(n + 1 - \pi_2)\cdots (n + 1-\pi_n)$,  $\pi^r=\pi_n\pi_{n-1}\cdots\pi_1$, and $\pi^i=\pi^i_1\pi^i_2\cdots\pi^i_n$ such that $\pi^i_j=\pi_m$ if $\pi_j=m$, for any $1\leq m \leq n$. Given a mesh pattern $(\tau,R)$, the corresponding operations are defined as
\[(\tau,R)^c=(\tau^c,R^c),\hspace{0.5cm}(\tau,R)^r=(\tau^r,R^r),\hspace{0.5cm}(\tau,R)^i=(\tau^i,R^i),
\]
where 
\[
R^c=\{(x,n-y): (x,y)\in R\},
\]
\[
R^r=\{(n-x,y): (x,y)\in R\},
\]
\[
R^i=\{(y,x): (x,y)\in R\}.
\]
For instance, for the mesh pattern $q=(123,\{\boks{0}{0},\boks{1}{2},\boks{2}{1},\boks{3}{1}\})$, the specific operations are detailed as follows:
\[
\pattern{scale=0.8}{3}{1/1,2/2,3/3}{0/0,1/2,2/1,3/1}\xrightarrow{c}\pattern{scale=0.8}{3}{1/3,2/2,3/1}{0/3,1/1,2/2,3/2}\xrightarrow{r}\pattern{scale=0.8}{3}{1/1,2/2,3/3}{0/2,1/2,2/1,3/3}\xrightarrow{i}\pattern{scale=0.8}{3}{1/1,2/2,3/3}{2/0,1/2,2/1,3/3}.
\]

Notice  that the shadings of these eight pairs exhibit symmetry both horizontally and vertically. 
This enables us to use the complement (or reverse) operation. 
To be precise, we let $f(\pi)=\pi^c$ (or $f(\pi)=\pi^r$) for any $\pi \in S_n$, which exchanges the occurrences of each pair of patterns.

\subsection{Results obtained by swapping two elements}\label{sub-dir-bijections}
In this subsection, we study the following ten pairs:\vspace{0.1cm}
\begin{center}
\begin{tabular}{rrrrr}
S9 = $\pattern{scale = 0.4}{3}{1/1,2/2,3/3}{0/0,0/1,0/2,0/3,1/0,1/1,1/2,1/3,2/0,2/1,2/2,2/3,3/0,3/1,3/2}\pattern{scale = 0.4}{3}{1/3,2/2,3/1}{0/0,0/1,0/2,0/3,1/0,1/1,1/2,1/3,2/0,2/1,2/2,2/3,3/0,3/1,3/2}$; & \hspace{-0.3cm}
S10 = $\pattern{scale = 0.4}{3}{1/1,2/2,3/3}{0/1,0/2,0/3,3/0,1/0,1/1,1/2,1/3,2/0,2/1,2/2,2/3,3/1,3/2,3/3}\pattern{scale = 0.4}{3}{1/3,2/2,3/1}{0/1,0/2,1/0,0/3,3/0,1/1,1/2,1/3,2/0,2/1,2/2,2/3,3/1,3/2,3/3}$;&\hspace{-0.3cm}
S11 = $\pattern{scale = 0.4}{3}{1/1,2/2,3/3}{0/0,0/1,0/2,0/3,1/0,1/1,1/2,1/3,2/0,2/1,2/3,3/0,3/1,3/2,3/3}\pattern{scale = 0.4}{3}{1/3,2/2,3/1} {0/0,0/1,0/2,0/3,1/0,1/1,1/2,1/3,2/0,2/1,2/3,3/0,3/1,3/2,3/3}$;&\hspace{-0.3cm}
S12 = $\pattern{scale = 0.4}{3}{1/1,2/2,3/3}{0/0,0/1,0/2,0/3,1/0,1/2,1/3,2/0,2/1,2/2,2/3,3/0,3/1,3/2,3/3}\pattern{scale = 0.4}{3}{1/3,2/2,3/1}{0/0,0/1,0/2,0/3,1/0,1/2,1/3,2/0,2/1,2/2,2/3,3/0,3/1,3/2,3/3}$; &\hspace{-0.3cm}
S13 = $\pattern{scale = 0.4}{3}{1/1,2/2,3/3}{0/0,0/1,0/2,0/3,1/0,1/1,1/3,2/0,2/2,2/3,3/0,3/1,3/2,3/3}\pattern{scale = 0.4}{3}{1/3,2/2,3/1}{0/0,0/1,0/2,0/3,1/0,1/1,1/3,2/0,2/2,2/3,3/0,3/1,3/2,3/3}$; \\[5pt] 
S14 = $\pattern{scale = 0.4}{3}{1/1,2/2,3/3}{0/0,0/1,0/2,0/3,1/0,1/2,1/3,2/0,2/1,2/3,3/0,3/1,3/2,3/3}\pattern{scale = 0.4}{3}{1/3,2/2,3/1}{0/0,0/1,0/2,0/3,1/0,1/2,1/3,2/0,2/1,2/3,3/0,3/1,3/2,3/3}$;&\hspace{-0.3cm}
S15 = $\pattern{scale = 0.4}{3}{1/1,2/2,3/3}{0/0,0/1,0/2,0/3,1/0,1/1,1/3,2/0,2/3,3/0,3/1,3/2,3/3}\pattern{scale = 0.4}{3}{1/3,2/2,3/1} {0/0,0/1,0/2,0/3,1/0,1/1,1/3,2/0,2/3,3/0,3/1,3/2,3/3}$; &\hspace{-0.3cm}
~S16 = $\pattern{scale = 0.4}{3}{1/1,2/2,3/3}{0/0,0/1,0/2,0/3,1/0,1/3,2/0,2/2,2/3,3/0,3/1,3/2,3/3}\pattern{scale = 0.4}{3}{1/3,2/2,3/1}{0/0,0/1,0/2,0/3,1/0,1/3,2/0,2/2,2/3,3/0,3/1,3/2,3/3}$; &\hspace{-0.3cm}
S17 = $\pattern{scale = 0.4}{3}{1/1,2/2,3/3}{0/0,0/1,0/2,0/3,1/0,1/1,1/3,2/0,2/2,2/3,3/0,3/1,3/2}\pattern{scale = 0.4}{3}{1/3,2/2,3/1} {0/0,0/1,0/2,0/3,1/0,1/1,1/3,2/0,2/2,2/3,3/0,3/1,3/2}$; &\hspace{-0.3cm}
S18 = $\pattern{scale = 0.4}{3}{1/1,2/2,3/3}{0/1,0/2,0/3,1/0,1/1,1/3,2/0,2/2,2/3,3/0,3/1,3/2,3/3}\pattern{scale = 0.4}{3}{1/3,2/2,3/1}{0/1,0/2,0/3,1/0,1/1,1/3,2/0,2/2,2/3,3/0,3/1,3/2,3/3}$.
\end{tabular}
\end{center}

Observe that for each of these 10 pairs, any permutation may contain occurrences of either $q_1$ or $q_2$, but never both simultaneously. Moreover, by applying complement and reverse operations to the patterns,
Nr.~S$9$, S$11$, S$13$, S$15$, and S$17$ have the same joint equidistribution with Nr.~S$10$, S$12$, S$14$, S$16$, and S$18$, respectively. 
We shall prove the joint equidistributions of the first half of these patterns by swapping 1 and another element in permutations.

\noindent 
{\footnotesize $\bullet$} For Nr.~S9, the unique occurrence of the pattern $q_1$ (resp.,   $q_2$) occurs in $ \pi \in S_n$ if and only if $\pi=123\pi_4\cdots \pi_n$ (resp.,   $321\pi_4\cdots\pi_n)$ for $n\geq 3$; otherwise $\pi$ avoids the patterns. Thus, let 
\[
f(\pi)= \left\{  
\begin{array}{ll}  
321\pi_4\cdots\pi_n, & \text{ } \pi = 123\pi_4\cdots\pi_n, \\
123\pi_4\cdots\pi_n, & \text{ } \pi = 321\pi_4\cdots\pi_n, \\
\pi, & \text{otherwise.}
\end{array}  
\right.
\]
\noindent
{\footnotesize $\bullet$} For Nr.~S11, the unique occurrence of the pattern $q_1$ (resp.,   $q_2$) occurs in $ \pi \in S_n$  if and only if $\pi=12\pi_3\cdots \pi_{n-1} n$ (resp.,   $n2 \pi_3\cdots \pi_{n-1} 1)$ for $n\geq 3$; otherwise, $\pi$ avoids the patterns. Thus, let
\[
f(\pi)= \left\{  
\begin{array}{ll}  
n2 \pi_3\cdots \pi_{n-1} 1, & \text{} \pi =12\pi_3\cdots \pi_{n-1} n, \\
12\pi_3\cdots \pi_{n-1} n, & \text{} \pi = n2 \pi_3\cdots \pi_{n-1} 1, \\
\pi, & \text{otherwise.}
\end{array}  
\right.
\]

\noindent
{\footnotesize $\bullet$} For Nr.~S13, the pattern $q_1$ (resp.,    $q_2$) occurs in $ \pi=\pi_1\pi_2\cdots \pi_n \in S_n$ if and only if $\pi_1=1$ and $\pi_n=n$ (resp.,   $\pi_1=n$ and $\pi_n=1$). Under this condition, the number of occurrences of $q_1$ or $q_2$ in $\pi$ is equal to the number of occurrences of the pattern $\pattern{scale = 0.5}{1}{1/1}{0/0,1/1}$ in the sequence $\pi_2\cdots \pi_{n-1}$. 
Hence, we let
\[
f(\pi)= \left\{  
\begin{array}{ll}  
n\pi_2 \cdots \pi_{n-1}1, & \text{} \pi =1\pi_2 \cdots \pi_{n-1} n, \\
1\pi_2 \cdots \pi_{n-1} n, & \text{} \pi = n\pi_2 \cdots  \pi_{n-1} 1, \\
\pi, & \text{otherwise.}
\end{array}  
\right.
\]
 \\[-3mm]


\noindent
{\footnotesize $\bullet$} For Nr.~S15,  the only difference  with Nr.~S13 is that the corresponding pattern in the sequence $\pi_2\cdots\pi_{n-1}$, which was $\pattern{scale = 0.5}{1}{1/1}{0/0,1/1}$ in Nr.~S13, is replaced by $\pattern{scale = 0.5}{1}{1/1}{0/0}$. 
Thus, we adopt the same map $f$ as in Nr.~S13. 

\noindent
{\footnotesize $\bullet$}  For Nr.~S17,  the structure of a permutation $\pi=\pi_1\pi_2\cdots \pi_n \in S_n$ with $k$ occurrences of $q_1$ (resp., $q_2$)  is shown on the left side (resp., right side) of Fig.~\ref{Nr.15-fig}. Namely, $\pi_1\pi_{j_1}\pi_t$, $\pi_1\pi_{j_2}\pi_t$, $\ldots$, and $\pi_1\pi_{j_k}\pi_t$ are $k$ occurrences of $q_1$ (resp., $q_2$). 
Clearly, $\pi_1 = 1$ for $q_1$ and $\pi_t = 1$ for $q_2$. Additionally, elements in the  area $B$ are not involved in any occurrence of these two patterns; otherwise, some $\pi_{j_i}'s$ (for $1 \leq i \leq k$) would be in the shaded areas. Hence, we let
\[
f(\pi)= \left\{  
\begin{array}{ll}  
\pi_t\cdots \pi_{i-1} 1 \pi_{t+1}\cdots\pi_n, & \text{} \pi =1\pi_2\cdots\pi_t \pi_{t+1} \cdots\pi_n, \\
1\pi_2\cdots\pi_{t-1} \pi_1 \pi_{t+1} \cdots\pi_n, & \text{} \pi = \pi_1 \pi_2 \cdots \pi_{t-1} 1 \pi_{t+1} \cdots\pi_n, \\
\pi, & \text{otherwise.}
\end{array}  
\right.
\]

 
\begin{figure}  
\begin{center}  
\begin{tabular}{ccc}
\begin{tikzpicture}[scale=0.7]

\tikzset{      
  grid/.style={        
    draw,        
    step=1cm,        
    gray!100,    
    thin,        
  },   
  cell/.style={      
    draw,      
    anchor=center,    
    text centered,    
  },    
  graycell/.style={   
    fill=gray!20,
    draw=none,     
    minimum width=1cm,   
    minimum height=1cm,     
    anchor=south west,            
  }    
}
      
 \fill[graycell] (0,0) rectangle (1,4);
 \fill[graycell] (0,5) rectangle (5,6);
 \fill[graycell] (1,0) rectangle (2,3);
 \fill[graycell] (1,4) rectangle (6,5);
 \fill[graycell] (2,0) rectangle (3,2);
 \fill[graycell] (2,3) rectangle (6,4);
 \fill[graycell] (3,0) rectangle (4,1);
 \fill[graycell] (3,2) rectangle (6,3);
 \fill[graycell] (4,1) rectangle (5,2);
 \fill[graycell] (5,0) rectangle (6,5);

\draw[grid] (0,0) grid (6,6);

  \node at (0.5,4.5) {\scriptsize$A_1$};
  \node at (1.5,3.5) {\scriptsize$A_2$};
  \node at (4.5,0.5) {\scriptsize$A_{k+1}$};
  \node at (5.5,5.5) {\scriptsize$B$};
  
  \node[anchor=center] at (0,-0.3) {\footnotesize{$\pi_1\vspace{0.2cm}=1$}}; 
  \node[anchor=east] at (1,3.7) {\footnotesize{$\pi_{j_1}$}}; 
  \node[anchor=east] at (2,2.7) {\footnotesize{$\pi_{j_2}$}}; 
  \node[anchor=east] at (4,0.7) {\footnotesize{$\pi_{j_k}$}}; 
  \node[anchor=east] at (5,4.7) {\footnotesize{$\pi_t$}}; 
 \node[anchor=center, rotate=-45] at (2.5,2.5) {$\ldots$};
 \node[anchor=center, rotate=-45] at (3.5,1.5) {$\ldots$};
 \draw[thick] (0,0) circle (5pt);
  \draw[thick] (0,0) circle (2pt);
  \filldraw[black] (1,4) circle (2.5pt); 
  \filldraw[black] (2,3) circle (2.5pt); 
  \filldraw[black] (4,1) circle (2.5pt); 
 \draw[thick] (5,5) circle (5pt);
  \draw[thick] (5,5) circle (2pt); 
 

\end{tikzpicture} 

	& \ & 
	
\begin{tikzpicture}[scale=0.7]
\tikzset{      
  grid/.style={        
    draw,        
    step=1cm,        
    gray!100,    
    thin,        
  },   
  cell/.style={      
    draw,      
    anchor=center,    
    text centered,    
  },    
  graycell/.style={   
    fill=gray!20,
    draw=none,     
    minimum width=1cm,   
    minimum height=1cm,     
    anchor=south west,            
  },  
  whitecell/.style={   
    fill=white,
    draw=none,     
    minimum width=1cm,   
    minimum height=1cm,     
    anchor=south west,            
  }    
}
      
\fill[graycell] (0,0) rectangle (1,4);
 \fill[graycell] (0,5) rectangle (5,6);
 \fill[graycell] (1,0) rectangle (2,3);
 \fill[graycell] (1,4) rectangle (6,5);
 \fill[graycell] (2,0) rectangle (3,2);
 \fill[graycell] (2,3) rectangle (6,4);
 \fill[graycell] (3,0) rectangle (4,1);
 \fill[graycell] (3,2) rectangle (6,3);
 \fill[graycell] (4,1) rectangle (5,2);
 \fill[graycell] (5,0) rectangle (6,5);

\draw[grid] (0,0) grid (6,6);

  \node at (0.5,4.5) {\scriptsize$A_1$};
  \node at (1.5,3.5) {\scriptsize$A_2$};
  \node at (4.5,0.5) {\scriptsize$A_{k+1}$};
  \node at (5.5,5.5) {\scriptsize$B$};
  
  \node[anchor=east] at (0,4.7) {\footnotesize{$\pi_1$}}; 
  \node[anchor=east] at (1,3.7) {\footnotesize{$\pi_{j_1}$}}; 
  \node[anchor=east] at (2,2.7) {\footnotesize{$\pi_{j_2}$}}; 
  \node[anchor=east] at (4,0.7) {\footnotesize{$\pi_{j_k}$}}; 
  \node[anchor=center] at (5,-0.3) {\footnotesize{$\pi_t=1$}}; 
 \node[anchor=center, rotate=-45] at (2.5,2.5) {$\ldots$};
 \node[anchor=center, rotate=-45] at (3.5,1.5) {$\ldots$};
   \draw[thick] (0,5) circle (5pt);
  \draw[thick] (0,5) circle (2pt);
  \filldraw[black] (1,4) circle (2.5pt); 
  \filldraw[black] (2,3) circle (2.5pt); 
  \filldraw[black] (4,1) circle (2.5pt); 
  \draw[thick] (5,0) circle (5pt);
  \draw[thick] (5,0) circle (2pt);

\end{tikzpicture} 
 \end{tabular}
\vspace{0.3cm}

\end{center}
\caption{The structure of Nr.~17}\label{Nr.15-fig}
\end{figure}
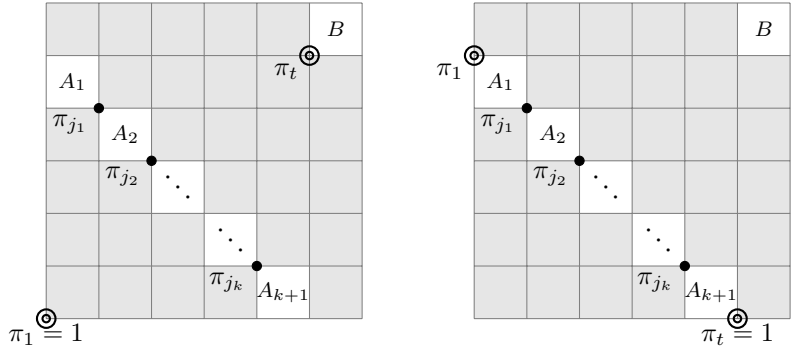

\subsection{Results proved via recurrence relations}\label{sub-recur-relat}
In this subsection, we consider the following two pairs:
\vspace{0.1cm}
\begin{center}
\begin{tabular}{rrrr}
S19 = $\pattern{scale = 0.4}{3}{1/1,2/2,3/3}{0/0,0/1,0/2,1/0,1/1,1/2,2/0,2/1,2/2}\pattern{scale = 0.4}{3}{1/3,2/2,3/1}{0/0,0/1,0/2,1/0,1/1,1/2,2/0,2/1,2/2}$;&\hspace{-0.3cm}
S20 = $\pattern{scale = 0.4}{3}{1/1,2/2,3/3}{1/1,1/2,1/3,2/1,2/2,2/3,3/1,3/2,3/3}\pattern{scale = 0.4}{3}{1/3,2/2,3/1}{1/1,1/2,1/3,2/1,2/2,2/3,3/1,3/2,3/3}$.
\end{tabular}
\end{center}

\begin{thm}\label{thm-pairs-19-20}
We have $\pattern{scale = 0.5}{3}{1/1,2/2,3/3}{0/0,0/1,0/2,1/0,1/1,1/2,2/0,2/1,2/2}\sim_{jd}\hspace{-1mm}\pattern{scale = 0.5}{3}{1/3,2/2,3/1}{0/0,0/1,0/2,1/0,1/1,1/2,2/0,2/1,2/2}$ and  $\hspace{-1mm}\pattern{scale = 0.5}{3}{1/1,2/2,3/3}{1/1,1/2,1/3,2/1,2/2,2/3,3/1,3/2,3/3}\sim_{jd}\hspace{-1mm}\pattern{scale = 0.5}{3}{1/3,2/2,3/1}{1/1,1/2,1/3,2/1,2/2,2/3,3/1,3/2,3/3}$
(Nr.~S19 and S20). 
\end{thm}
\begin{proof}
 
It suffices to prove the joint equidistribution of Nr.~S19, since Nr.~S20 can be derived from Nr.~S19 via complement and reverse operations as follows:
\[
\pattern{scale = 0.5}{3}{1/1,2/2,3/3}{0/0,0/1,0/2,1/0,1/1,1/2,2/1,2/2,2/0}\xrightarrow{c}\pattern{scale = 0.5}{3}{1/3,2/2,3/1}{0/3,0/1,0/2,1/3,1/1,1/2,2/1,2/2,2/3}\xrightarrow{r}\pattern{scale = 0.5}{3}{1/1,2/2,3/3}{3/3,3/1,3
/2,1/3,1/1,1/2,2/1,2/2,2/3}, \hspace{0.5cm}\pattern{scale = 0.5}{3}{1/3,2/2,3/1}{0/0,0/1,0/2,1/0,1/1,1/2,2/1,2/2,2/0}\xrightarrow{c}\pattern{scale = 0.5}{3}{1/1,2/2,3/3}{0/3,0/1,0/2,1/3,1/1,1/2,2/1,2/2,2/3}\xrightarrow{r}\pattern{scale = 0.5}{3}{1/3,2/2,3/1}{3/3,3/1,3
/2,1/3,1/1,1/2,2/1,2/2,2/3}.
\]

\begin{figure}  
\begin{center}  
\begin{tabular}{ccc}
\begin{tikzpicture}[scale=0.7]

\tikzset{      
   grid/.style={        
      draw,        
      step=1cm,        
      gray!100,    
      thin,        
    },   
    cell/.style={      
      draw,      
      anchor=center,    
      text centered,    
     },    
    graycell/.style={   
      fill=gray!20,
      draw=none,     
      minimum width=1cm,   
      minimum height=1cm,     
      anchor=south west,            
    }    
  } 
   \fill[graycell] (0,0) rectangle (1,7);
  \fill[graycell] (1,0) rectangle (2,6);
  \fill[graycell] (2,0) rectangle (3,5);
  \fill[graycell] (3,0) rectangle (4,3);
  \fill[graycell] (4,0) rectangle (5,3);
  \fill[graycell] (5,0) rectangle (6,2);
  \fill[graycell] (6,0) rectangle (7,1);

 \draw[grid] (0,0) grid (8,8);  

\fill[white] (3,3) rectangle (5,5);
\fill[graycell] (3,3) rectangle (4.4,4.4);
\draw[draw=gray!60, line width=1pt] (3,3)--(3,5)--(5,5)--(5,3)--(3,3) ;
  \node at (0.5,7.5) {\scriptsize$A_1$};  
  \node at (1.5,6.5) {\scriptsize$A_2$};  
  \node at (4.5,3.5) {\scriptsize$A_i$};  
  \node at (7.5,0.5) {\scriptsize$A_s$};
  \node[anchor=east] at (0,6.7) {\footnotesize{$x_1$}}; 
  \node[anchor=east] at (1,5.7) {\footnotesize{$x_2$}}; 
  \node[anchor=east] at (3,2.7) {\footnotesize{$x_i$}}; 
  \node[anchor=center] at (3.3,3.9) {\footnotesize{\scriptsize$a$}}; 
  \node[anchor=center] at (4.5,4.8) {\footnotesize{\scriptsize$b$}}; 
  \node[anchor=center] at (4.5,1.7) {\footnotesize{\scriptsize$x_{i+1}$}}; 
  \node[anchor=west] at (7,-0.3) {\footnotesize{$x_s=1$}}; 
 
  \filldraw[black] (0,7) circle (2.5pt);
   \filldraw[black] (1,6) circle (2.5pt);
   \draw[thick] (3,3) circle (5pt);
   \draw[thick] (3,3) circle (2pt); 
  \draw[thick] (3.6,3.6) circle (5pt);
  \draw[thick] (3.6,3.6) circle (2pt);
  \draw[thick] (4.4,4.4) circle (5pt);
  \draw[thick] (4.4,4.4) circle (2pt);
  \filldraw[black] (5,2) circle (2.5pt); 
  \filldraw[black] (7,0) circle (2.5pt);  
 \node[anchor=center, rotate=-45] at (2.5,5.5) {$\ldots$};
 \node[anchor=center, rotate=-45] at (6.5,1.5) {$\ldots$};
\end{tikzpicture} 
	& \ & 
\begin{tikzpicture}[scale=0.7]

\tikzset{      
   grid/.style={        
      draw,        
      step=1cm,        
      gray!100,    
      thin,        
    },   
    cell/.style={      
      draw,      
      anchor=center,    
      text centered,    
     },    
    graycell/.style={   
      fill=gray!20,
      draw=none,     
      minimum width=1cm,   
      minimum height=1cm,     
      anchor=south west,            
    }    
  } 
      
  \fill[graycell] (0,0) rectangle (1,7);
  \fill[graycell] (1,0) rectangle (2,6);
  \fill[graycell] (2,0) rectangle (3,5);
  \fill[graycell] (3,0) rectangle (4,4);
  \fill[graycell] (4,0) rectangle (5,4);
  \fill[graycell] (5,0) rectangle (6,2);
  \fill[graycell] (6,0) rectangle (7,1);

 \draw[grid] (0,0) grid (8,8);  

  \node at (1.5,6.5) {\scriptsize$B_1$};  
  \node at (4.5,3.5) {\scriptsize$B_i$};  
  
  \node[anchor=east] at (0,6.7) {\footnotesize{$x_1$}}; 
  \node[anchor=east] at (1,5.7) {\footnotesize{$x_2$}}; 
  \node[anchor=east] at (3,3.7) {\footnotesize{$x_i$}}; 
  \node[anchor=center] at (3.5,2.7) {\footnotesize{\scriptsize$x_{i+1}$}}; 
  \node[anchor=center] at (4.5,1.7) {\footnotesize{\scriptsize$x_{i+2}$}}; 
  \node[anchor=west] at (7,-0.3) {\footnotesize{$x_s=1$}}; 
 
  \filldraw[black] (0,7) circle (2.5pt);
   \filldraw[black] (1,6) circle (2.5pt);
   \draw[thick] (3,4) circle (5pt);
  \draw[thick] (3,4) circle (2pt);
   
  \draw[thick] (4,3) circle (5pt);
  \draw[thick] (4,3) circle (2pt);
  \draw[thick] (5,2) circle (5pt);
  \draw[thick] (5,2) circle (2pt);
  
  \filldraw[black] (7,0) circle (2.5pt);  
 \node[anchor=center, rotate=-45] at (2.5,5.5) {$\ldots$};
 \node[anchor=center, rotate=-45] at (6.5,1.5) {$\ldots$};
\end{tikzpicture} 
 \end{tabular}
\vspace{-0.5cm}
\end{center}
\caption{The structure of $n$-permutations containing occurrences of $q_1$ and $q_2$ in Nr.~S19, where the $x_i$'s represent the left-to-right minima in permutations.}\label{bo-fig}
\end{figure}
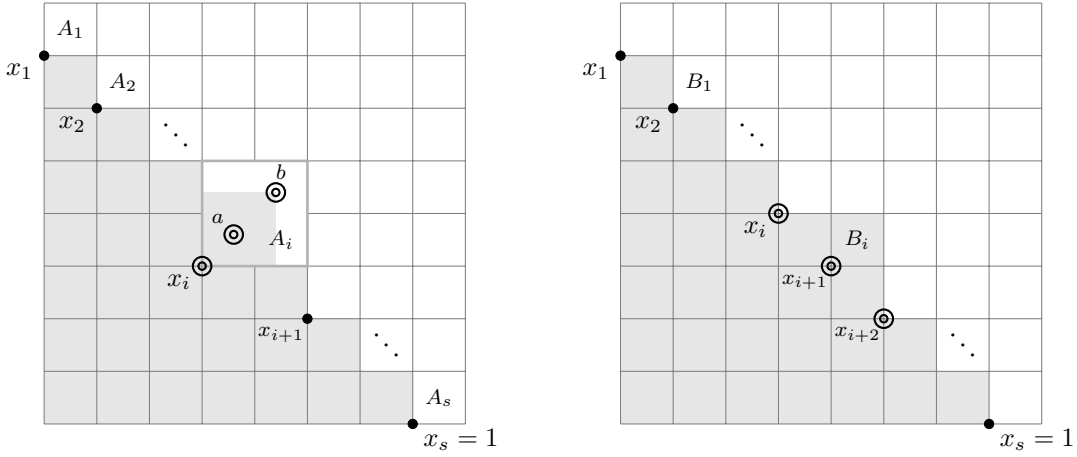
\noindent Here, we let
\[q_1=\pattern{scale = 0.5}{3}{1/1,2/2,3/3}{0/0,0/1,0/2,1/0,1/1,1/2,2/0,2/1,2/2},\hspace{0.6cm}q_2=\pattern{scale = 0.5}{3}{1/3,2/2,3/1}{0/0,0/1,0/2,1/0,1/1,1/2,2/0,2/1,2/2}.
\]

Recall that for a  permutation $\pi = \pi_1 \cdots \pi_n \in S_n$, a letter $\pi_i$ is said to be  a {\em left-to-right minimum} if $\pi_i < \mbox{min}(\pi_1 \cdots \pi_{i-1})$. In particular, $\pi_1$ is always a left-to-right minimum in $\pi$. For instance, in the permutation $\pi=45123$, elements 4 and 1 are left-to-right minima. We label the slots from left to right with the $(i+1)$-th slot located between $\pi_{i}$ and $\pi_{i+1}$. In particular, the first slot (resp., the last slot) refers to the position immediately before the first entry (resp., after the last entry). Denote
\[
T^{(1)}_{n,k,\ell}:=\# \{\pi \in \A(n,k,\ell):  \pi_1>\pi_2 \},
\]
\[
T^{(2)}_{n,k,\ell}:=\# \{\pi \in \A(n,k,\ell): \pi_1<\pi_2 \}.
\]
It is clear that 
\begin{align}\label{HK12}
T_{n,k,\ell}=T^{(1)}_{n,k,\ell}+T^{(2)}_{n,k,\ell}.
\end{align}
To prove $T_{n,k,\ell}=T_{n,\ell,k}$, we shall show
\[
T^{(1)}_{n,k,\ell}=T^{(2)}_{n,\ell,k}\text{ and } 
T^{(2)}_{n,k,\ell}=T^{(1)}_{n,\ell,k}.
\]
 We claim that
\begin{align}
T^{(1)}_{n,k,\ell}&=(n-2)T^{(1)}_{n-1,k,\ell}+T^{(1)}_{n-1,k,\ell-1}+T^{(2)}_{n-1,k,\ell},\label{A1_n,k,l}\\
T^{(2)}_{n,k,\ell}&=T^{(1)}_{n-1,k,\ell}+(n-2)T^{(2)}_{n-1,k,\ell}+T^{(2)}_{n-1,k-1,\ell}. \label{A2_n,k,l}
\end{align}
As shown on the left side of Fig.~\ref{bo-fig}, $x_iab$ is an occurrence of $q_1$, where the first element is a left-to-right minimum and the second and third elements are in the area 
$A_i's$, which has been observed in~\cite[Figure~1]{KLv}. 
Referring to the right side of Fig.~\ref{bo-fig}, any occurrence of $q_2$ consists of three consecutive left-to-right minima with $B_i=\emptyset$, as exemplified by $x_ix_{i+1}x_{i+2}$.

To generate a permutation $\pi=\pi_1\pi_2\cdots \pi_n \in S_n$ counted by the left-hand side of the two equations, we consider inserting the largest element $n$ into some $\pi'=\pi'_1\pi'_2\cdots \pi'_n \in S_{n-1}$. The items in~(\ref{A1_n,k,l}) are as follows:
\begin{itemize}
\item $\pi'_1>\pi'_2$. Then inserting $n$ in any slot except the second results in $\pi_1>\pi_2$. Namely,
in any of the last $(n-2)$ slots, the number of occurrences of both patterns would be preserved, which is counted by $(n-2)T^{(1)}_{n-1,k,\ell}$. However, inserting $n$ in the first slot establishes $n$ as a left-to-right minimum with $B_1=\emptyset$, thereby creating an extra occurrence of $q_2$ in the form of $n\pi_2\pi_3$, without affecting other occurrences. Thus, this case contributes to $T^{(1)}_{n-1,k,\ell-1}$.
\item $\pi'_1<\pi'_2$. Then inserting $n$ only in the first slot makes $\pi_1>\pi_2$. This insertion does not affect any occurrence of either pattern and gives $T^{(2)}_{n-1,k,\ell}$.
 \end{itemize}
 For~(\ref{A2_n,k,l}), we divide our discussion into the following cases.
 \begin{itemize}
\item  $\pi'_1>\pi'_2$. Then inserting $n$ only in the second slot makes $\pi_1<\pi_2$, and this insertion does not change the number of occurrences of either pattern. Therefore, this case is given by $T^{(1)}_{n-1,k,\ell}$.
\item $\pi'_1<\pi'_2$. Then inserting $n$ in any slot except the first ensures $\pi_1<\pi_2$. Namely, in any of the second or the last $(n-3)$ slots, the number of occurrences of both patterns remains unchanged, which contributes to $(n-2)T^{(2)}_{n-1,k,\ell}$.
 Inserting $n$ in the third slot leads to $\pi_2, n \in A_1$, thereby creating an occurrence of $q_1$ in the form of $\pi_1\pi_2n$, without affecting other occurrences of either pattern. Thus, this case contributes to $T^{(2)}_{n-1,k-1,\ell}$.
\end{itemize}  
The initial conditions are $T^{(1)}_{2,0,0}=T^{(2)}_{2,0,0}=1$,  $T^{(1)}_{3,0,1}=T^{(2)}_{3,1,0}=1$ and $T^{(1)}_{3,1,0}=T^{(2)}_{3,0,1}=0$.
 By directly using~(\ref{A1_n,k,l}) and~(\ref{A2_n,k,l}), and applying induction on $n$, $k$ and $\ell$, we obtain the results. Thus, we complete the proof. 
\end{proof}
 
\subsection{Wilf-equivalence of Nr.~S21 and S22}\label{sub-Wilf}

The computer evidence also suggests two additional pairs of patterns with symmetric shadings.
\begin{conj} \label{con-pairs-21-22}
 $\pattern{scale = 0.5}{3}{1/1,2/2,3/3}{0/0,0/1,0/2,1/0,1/1,1/2,2/0,2/1,2/2,3/3}\sim_{jd} \hspace{-1mm}\pattern{scale = 0.5}{3}{1/3,2/2,3/1}{0/0,0/1,0/2,1/0,1/1,1/2,2/0,2/1,2/2,3/3}$ and $\hspace{-1mm}\pattern{scale = 0.5}{3}{1/1,2/2,3/3}{0/0,1/1,1/2,1/3,2/1,2/2,2/3,3/1,3/2,3/3}\sim_{jd}\hspace{-1mm}\pattern{scale = 0.5}{3}{1/3,2/2,3/1}{0/0,1/1,1/2,1/3,2/1,2/2,2/3,3/1,3/2,3/3}$  (Nr.~S21 and S22) .
\end{conj}

The patterns in pair Nr.~S22 can be obtained from those of Nr.~S21 by applying 
the reverse and complement operations. Therefore, it suffices to prove the 
joint equidistribution for Nr.~S21.
The Wilf-equivalence of the patterns in Conjecture \ref{con-pairs-21-22} holds. 
\begin{thm}\label{thm-pairs-21-22}
 The patterns $\pattern{scale = 0.5}{3}{1/1,2/2,3/3}{0/0,0/1,0/2,1/0,1/1,1/2,2/0,2/1,2/2,3/3}, \hspace{-1mm}\pattern{scale = 0.5}{3}{1/3,2/2,3/1}{0/0,0/1,0/2,1/0,1/1,1/2,2/0,2/1,2/2,3/3}, \hspace{-1mm}\pattern{scale = 0.5}{3}{1/1,2/2,3/3}{0/0,1/1,1/2,1/3,2/1,2/2,2/3,3/1,3/2,3/3}, \hspace{-1mm}\pattern{scale = 0.5}{3}{1/3,2/2,3/1}{0/0,1/1,1/2,1/3,2/1,2/2,2/3,3/1,3/2,3/3}$ are Wilf-equivalent.
\end{thm}

\begin{proof}
 Applying the complement and reverse operations, we have
\[
\pattern{scale = 0.5}{3}{1/1,2/2,3/3}{0/0,0/1,0/2,1/0,1/1,1/2,2/1,2/2,2/0,3/3}\xrightarrow{c}\pattern{scale = 0.5}{3}{1/3,2/2,3/1}{0/3,0/1,0/2,1/3,1/1,1/2,2/1,2/2,2/3,3/0}\xrightarrow{r}\pattern{scale = 0.5}{3}{1/1,2/2,3/3}{0/0,3/3,3/1,3
/2,1/3,1/1,1/2,2/1,2/2,2/3},\hspace{0.5cm} \pattern{scale = 0.5}{3}{1/3,2/2,3/1}{0/0,0/1,0/2,1/0,1/1,1/2,2/1,2/2,2/0,3/3}\xrightarrow{c}\pattern{scale = 0.5}{3}{1/1,2/2,3/3}{0/3,0/1,0/2,1/3,1/1,1/2,2/1,2/2,2/3,3/0}\xrightarrow{r}\pattern{scale = 0.5}{3}{1/3,2/2,3/1}{0/0,3/3,3/1,3
/2,1/3,1/1,1/2,2/1,2/2,2/3}.
\]
Thus, it suffices to prove the Wilf-equivalence of Nr.~S21. Here, we let
\[q_1=\pattern{scale = 0.5}{3}{1/1,2/2,3/3}{0/0,0/1,0/2,1/0,1/1,1/2,2/0,2/1,2/2,3/3},\hspace{0.6cm}q_2=\pattern{scale = 0.5}{3}{1/3,2/2,3/1}{0/0,0/1,0/2,1/0,1/1,1/2,2/0,2/1,2/2,3/3}.
\]
 We establish a  bijection $f$ that maps $S_n(q_1)$ to $S_n(q_2)$ as follows: If $\pi \in S_n(q_1) \cap S_n(q_2)$, then let $f(\pi)=\pi$. If $\pi \in S_n(q_1)$ and contains at least one occurrence of $q
_2$, then we apply the following swapping operation. 
\begin{itemize}
    \item Find the lexicographically first occurrence of $q_2$ and swap the first and third elements of it, as shown in Fig.~\ref{fig-pair21}. It is clear  the occurrence of $q_2$ turns to be an occurrence of $q_1$.
\end{itemize}
\noindent
Repeating the swapping operations until there are no more occurrences of $q_2$, we obtain a sequence of permutations $\pi= \pi^{(0)}, \pi^{(1)}, \dots, \pi^{(j)}$, we let $f(\pi)=\pi^{(j)}$. 
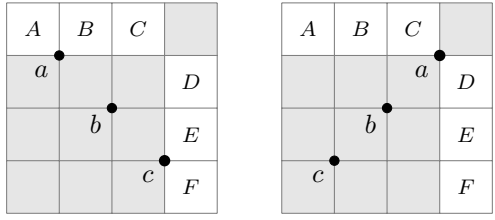
\begin{figure}[h]
\begin{center}  
\begin{tabular}{ccc}
\begin{tikzpicture}[scale=0.7]

\tikzset{      
   grid/.style={        
      draw,        
      step=1cm,        
      gray!100,    
      thin,        
    },   
    cell/.style={      
      draw,      
      anchor=center,    
      text centered,    
     },    
    graycell/.style={   
      fill=gray!20,
      draw=none,     
      minimum width=1cm,   
      minimum height=1cm,     
      anchor=south west,            
    }    
  } 
    \fill[graycell] (0,0) rectangle (3,3);
  \fill[graycell] (3,3) rectangle (4,4);
 
 \draw[grid] (0,0) grid (4,4);  
 
  \node at (0.5,3.5) {\scriptsize$A$};  
  \node at (1.5,3.5) {\scriptsize$B$};  
  \node at (2.5,3.5) {\scriptsize$C$};  
  \node at (3.5,2.5) {\scriptsize$D$};
  \node at (3.5,1.5) {\scriptsize$E$};
  \node at (3.5,0.5) {\scriptsize$F$};
  \node[anchor=east] at (3,0.7) {\footnotesize{$c$}}; 
  \node[anchor=east] at (2,1.7) {\footnotesize{$b$}}; 
  \node[anchor=east] at (1,2.7) {\footnotesize{$a$}}; 
 
  \filldraw[black] (1,3) circle (2.5pt);
   \filldraw[black] (2,2) circle (2.5pt);
   \filldraw[thick] (3,1) circle (2.5pt);

\end{tikzpicture} 
	& \ & 
\begin{tikzpicture}[scale=0.7]

\tikzset{      
   grid/.style={        
      draw,        
      step=1cm,        
      gray!100,    
      thin,        
    },   
    cell/.style={      
      draw,      
      anchor=center,    
      text centered,    
     },    
    graycell/.style={   
      fill=gray!20,
      draw=none,     
      minimum width=1cm,   
      minimum height=1cm,     
      anchor=south west,            
    }    
  } 
      
  \fill[graycell] (0,0) rectangle (3,3);
  \fill[graycell] (3,3) rectangle (4,4);
 
 \draw[grid] (0,0) grid (4,4);  

  \node at (0.5,3.5) {\scriptsize$A$};  
  \node at (1.5,3.5) {\scriptsize$B$};  
  \node at (2.5,3.5) {\scriptsize$C$};  
  \node at (3.5,2.5) {\scriptsize$D$};
  \node at (3.5,1.5) {\scriptsize$E$};
  \node at (3.5,0.5) {\scriptsize$F$};
  \node[anchor=east] at (1,0.7) {\footnotesize{$c$}}; 
  \node[anchor=east] at (2,1.7) {\footnotesize{$b$}}; 
  \node[anchor=east] at (3,2.7) {\footnotesize{$a
  $}}; 
 
  \filldraw[black] (1,1) circle (2.5pt);
   \filldraw[black] (2,2) circle (2.5pt);
   \filldraw[thick] (3,3) circle (2.5pt);
\end{tikzpicture} 
 \end{tabular}
\vspace{-0.5cm}

\end{center}
\caption{The swapping operation that maps $\pi^{(i)}$ with the lexicographically first occurrence of $q_2$  (on the left side) to $\pi^{(i+1)}$ with one occurrence of $q_1$ (on the right side) in Theorem \ref{thm-pairs-21-22}.}\label{fig-pair21}
\end{figure}

\textbf{$f$ is well-defined.} 
We shall prove the sequence described above terminates in finite steps.
Let $abc$ be the lexicographically first occurrence of $q_2$ in $\pi^{(i)}$ as shown on the left side of Fig.~\ref{fig-pair21}.
We claim that no new occurrence of $q_2$ is created in $\pi^{(i+1)}$ that is lexicographically smaller than the occurrence $abc$.
Suppose, contrary to our claim, that  $a'b'c'$ is such an occurrence of $q_2$ in $\pi^{(i+1)}$. Neither a nor b can appear in $a'b'c'$, since otherwise $a'b'c'$ would not be an occurrence of $q_2$ due to the existence of $c$. 
\begin{itemize}
    \item If $c=a'$, then $b'$ and $c'$  appear in the  area $F$, which would force $cb'c'$ to be the right of $abc$, contrary to the assumption that $cb'c'$ is lexicographically smaller than the occurrence $abc$.
    \item  If $c=b'$, then $a'$ must be in area $A$ and $c'$ must be in area $F$. Therefore, $a'cc'$ would not be an occurrence of $q_2$ since $a$ and $b$ are in shaded areas.
    \item  If $c = c'$, then the elements $a'$ and $b'$ are in the area~$A$.
In this case, $a'b'a$ would be an occurrence of $q_2$ to the left of $abc$ in $\pi^{(i)}$, which is a contradiction.
\end{itemize}
This proves our claim. Thus $f(\pi)=\pi^{(j)} \in S_n(q_2)$. 

\textbf{$f$ is reversible.} We consider the inverse mapping $f^{-1}$. For any $\sigma \in S_n(q_1) \cap S_n(q_2)$, we let $f^{-1}(\sigma)=\sigma$. For $\sigma\in S_n(q_2)$ that contains occurrences of $q_1$, find the lexicographically first $q_1$ and swap the first and third elements of it. Along similar lines, we can see that the procedure never creates any new occurrences of $q_1$ that are lexicographically smaller than the lexicographically first one. 
\end{proof}

\section{Minus-antipodal shadings}\label{sec-an-sy}

In this section, we consider 36 joint equidistributions of patterns with minus-antipodal shadings, as listed in Table~\ref{tab-anti-shadings}. In Sect.~\ref{subsec-an-cr}, we obtain the results by complement, reverse, and inverse operations. In Sect.~\ref{sub-an-rl}, we prove the results via recurrence relations. 

\subsection{Results obtained by complement and reverse}\label{subsec-an-cr}
We consider the following 16 pairs in this subsection.
\vspace{0.1cm}
\begin{center}
\begin{tabular}{rrrrr}
A1 = $\pattern{scale = 0.4}{3}{1/1,2/2,3/3}{0/0,0/1,0/2,0/3,1/1,1/2,3/1,3/2}\pattern{scale = 0.4}{3}{1/3,2/2,3/1}{0/0,0/1,0/2,0/3,1/1,1/2,3/1,3/2}$;&\hspace{-0.3cm}
A2 = $\pattern{scale = 0.4}{3}{1/1,2/2,3/3}{0/1,0/2,2/1,2/2,3/0,3/1,3/2,3/3}\pattern{scale = 0.4}{3}{1/3,2/2,3/1}{0/1,0/2,2/1,2/2,3/0,3/1,3/2,3/3}$;  &\hspace{-0.3cm}
A3 = $\pattern{scale = 0.4}{3}{1/1,2/2,3/3}{0/0,1/0,2/0,3/0,1/1,2/1,1/3,2/3}\pattern{scale = 0.4}{3}{1/3,2/2,3/1}{0/0,1/0,2/0,3/0,1/1,2/1,1/3,2/3}$; & \hspace{-0.3cm}
A4 = $\pattern{scale = 0.4}{3}{1/1,2/2,3/3}{1/0,1/2,2/0,2/2,0/3,1/3,2/3,3/3}\pattern{scale = 0.4}{3}{1/3
,2/2,3/1}{1/0,1/2,2/0,2/2,0/3,1/3,2/3,3/3}$;\\[5pt]
A5 = $\pattern{scale = 0.4}{3}{1/1,2/2,3/3}{0/0,0/1,0/2,0/3,2/1,2/2,3/1,3/2}\pattern{scale = 0.4}{3}{1/3,2/2,3/1}{0/0,0/1,0/2,0/3,2/1,2/2,3/1,3/2}$;   &\hspace{-0.3cm}
A6 = $\pattern{scale = 0.4}{3}{1/1,2/2,3/3}{0/1,0/2,1/1,1/2,3/0,3/1,3/2,3/3}\pattern{scale = 0.4}{3}{1/3,2/2,3/1}{0/1,0/2,1/1,1/2,3/0,3/1,3/2,3/3}$;  &\hspace{-0.3cm}
A7 = $\pattern{scale = 0.4}{3}{1/1,2/2,3/3}{0/0,1/0,2/0,3/0,1/2,2/2,1/3,2/3}\pattern{scale = 0.4}{3}{1/3,2/2,3/1}{0/0,1/0,2/0,3/0,1/2,2/2,1/3,2/3}$; & \hspace{-0.3cm}
A8 = $\pattern{scale = 0.4}{3}{1/1,2/2,3/3}{1/0,1/1,2/0,2/1,0/3,1/3,2/3,3/3}\pattern{scale = 0.4}{3}{1/3,2/2,3/1}{1/0,1/1,2/0,2/1,0/3,1/3,2/3,3/3}$;\\[5pt]
A9 = $\pattern{scale = 0.4}{3}{1/1,2/2,3/3}{1/0,1/1,1/2,1/3,2/0,2/3,3/0,3/3}\pattern{scale = 0.4}{3}{1/3,2/2,3/1}{1/0,1/1,1/2,1/3,2/0,2/3,3/0,3/3}$;  & \hspace{-0.3cm}
A10 = $\pattern{scale = 0.4}{3}{1/1,2/2,3/3}{0/0,0/3,1/0,1/3,2/0,2/1,2/2,2/3}\pattern{scale = 0.4}{3}{1/3
,2/2,3/1}{0/0,0/3,1/0,1/3,2/0,2/1,2/2,2/3}$; &\hspace{-0.3cm}
A11 = $\pattern{scale = 0.4}{3}{1/1,2/2,3/3}{0/1,0/2,0/3,1/1,2/1,3/1,3/2,3/3}\pattern{scale = 0.4}{3}{1/3,2/2,3/1}{0/1,0/2,0/3,1/1,2/1,3/1,3/2,3/3}$; &\hspace{-0.3cm}
A12 = $\pattern{scale = 0.4}{3}{1/1,2/2,3/3}{0/0,0/1,0/2,1/2,2/2,3/0,3/1,3/2}\pattern{scale = 0.4}{3}{1/3,2/2,3/1}{0/0,0/1,0/2,1/2,2/2,3/0,3/1,3/2}$; 
\\[5pt]
A13 = $\pattern{scale = 0.4}{3}{1/1,2/2,3/3}{0/0,0/3,1/0,1/1,1/2,1/3,2/0,2/3}\pattern{scale = 0.4}{3}{1/3,2/2,3/1}{0/0,0/3,1/0,1/1,1/2,1/3,2/0,2/3}$; & \hspace{-0.3cm}
A14 = $\pattern{scale = 0.4}{3}{1/1,2/2,3/3}{1/0,2/0,3/0,2/1,2/2,1/3,2/3,3/3}\pattern{scale = 0.4}{3}{1/3,2/2,3/1}{1/0,2/0,3/0,2/1,2/2,1/3,2/3,3/3}$;&\hspace{-0.3cm}
A15 = $\pattern{scale = 0.4}{3}{1/1,2/2,3/3}{0/0,0/1,0/2,1/1,2/1,3/0,3/1,3/2}\pattern{scale = 0.4}{3}{1/3,2/2,3/1}{0/0,0/1,0/2,1/1,2/1,3/0,3/1,3/2}$; & \hspace{-0.3cm}
A16 = $\pattern{scale = 0.4}{3}{1/1,2/2,3/3}{0/1,0/2,0/3,1/2,2/2,3/1,3/2,3/3}\pattern{scale = 0.4}{3}{1/3
,2/2,3/1}{0/1,0/2,0/3,1/2,2/2,3/1,3/2,3/3}$.
\end{tabular}
\end{center}  
Observe that the two patterns in each pair in the first two columns can be obtained from each other via the complement operation, whereas those in the last two columns can be obtained via the reverse operation. This proves the joint equidistributions of each pair. 
Moreover, pairs in the same row are also given the same joint distributions. Here, we demonstrate the operation for the first four pairs (A1$\to$A2$\to$A4$\to$A3) as an example, and the remaining  pairs are obtained in the same manner,
 \[
 \begin{tabular}{c}
 \pattern{scale = 0.4}{3}{1/1,2/2,3/3}{0/0,0/1,0/2,0/3,1/1,1/2,3/1,3/2}
 \\
 \vspace{4pt}
 \footnotesize $A1(p_1)$
 \end{tabular}
 \raisebox{8pt}{$\xrightarrow{r}$}  
 \begin{tabular}{c}
 \pattern{scale = 0.4}{3}{1/3,2/2,3/1}{0/1,0/2,2/1,2/2,3/0,3/1,3/2,3/3} \\
 \vspace{4pt}
 \footnotesize $A2 (p_2)$
 \end{tabular}
 \raisebox{8pt}{$\xrightarrow{i}$}  
 \begin{tabular}{c}
 \pattern{scale = 0.4}{3}{1/3,2/2,3/1}{1/0,1/2,2/0,2/2,0/3,1/3,2/3,3/3} \\
 \vspace{4pt}
 \footnotesize $A4 (p_2)$
 \end{tabular}
 \raisebox{8pt}{$\xrightarrow{c}$}  
 \begin{tabular}{c}
 \pattern{scale = 0.4}{3}{1/1,2/2,3/3}{0/0,1/0,2/0,3/0,1/1,2/1,1/3,2/3} \\
 \vspace{4pt}
 \footnotesize $A3 (p_1)$
 \end{tabular}
 \]
 \[
 \begin{tabular}{c}
 \pattern{scale = 0.4}{3}{1/3,2/2,3/1}{0/0,0/1,0/2,0/3,1/1,1/2,3/1,3/2} \\
 \vspace{4pt}
 \footnotesize $A1(p_2)$
 \end{tabular}
 \raisebox{8pt}{$\xrightarrow{r}$}  
 \begin{tabular}{c}
 \pattern{scale = 0.4}{3}{1/1,2/2,3/3}{0/1,0/2,2/1,2/2,3/0,3/1,3/2,3/3} \\
 \vspace{4pt}
 \footnotesize $A2 (p_1)$
 \end{tabular}
 \raisebox{8pt}{$\xrightarrow{i}$}  
 \begin{tabular}{c}
 \pattern{scale = 0.4}{3}{1/1,2/2,3/3}{1/0,1/2,2/0,2/2,0/3,1/3,2/3,3/3} \\
 \vspace{4pt}
 \footnotesize $A4 (p_1)$
 \end{tabular}
 \raisebox{8pt}{$\xrightarrow{c}$}  
 \begin{tabular}{c}
 \pattern{scale = 0.4}{3}{1/3,2/2,3/1}{0/0,1/0,2/0,3/0,1/1,2/1,1/3,2/3} \\
 \vspace{4pt}
 \footnotesize $A3 (p_2)$
 \end{tabular}
 \]

\subsection{Results proved via recurrence relations}\label{sub-an-rl}
 We consider the following 20 pairs:\vspace{0.1cm}
\begin{center}
\begin{tabular}{rrrr}
A17 = $\pattern{scale = 0.4}{3}{1/1,2/2,3/3}{0/0,0/1,0/2,0/3,1/1,2/1,3/1,3/2}\pattern{scale = 0.4}{3}{1/3,2/2,3/1}{0/0,0/1,0/2,0/3,1/1,2/1,3/1,3/2}$;&\hspace{-0.3cm}
A18 = $\pattern{scale = 0.4}{3}{1/1,2/2,3/3}{0/0,0/1,0/2,0/3,1/2,2/2,3/2,3/1}\pattern{scale = 0.4}{3}{1/3,2/2,3/1}{0/0,0/1,0/2,0/3,1/2,2/2,3/2,3/1}$; &\hspace{-0.3cm}
A19 = $\pattern{scale = 0.4}{3}{1/1,2/2,3/3}{0/0,1/0,2/0,3/0,1/1,1/2,1/3,2/3}\pattern{scale = 0.4}{3}{1/3,2/2,3/1}{0/0,1/0,2/0,3/0,1/1,1/2,1/3,2/3}$; & \hspace{-0.3cm}
A20 = $\pattern{scale = 0.4}{3}{1/1,2/2,3/3}{0/0,1/0,2/0,3/0,2/1,2/2,2/3,1/3}\pattern{scale = 0.4}{3}{1/3,2/2,3/1}{0/0,1/0,2/0,3/0,2/1,2/2,2/3,1/3}$; \\[5pt]
A21 = $\pattern{scale = 0.4}{3}{1/1,2/2,3/3}{0/1,0/2,1/1,2/1,3/1,3/0,3/2,3/3}\pattern{scale = 0.4}{3}{1/3,2/2,3/1}{0/1,0/2,1/1,2/1,3/1,3/0,3/2,3/3}$; &\hspace{-0.3cm}
A22 = $\pattern{scale = 0.4}{3}{1/1,2/2,3/3}{0/1,0/2,1/2,2/2,3/2,3/1,3/0,3/2,3/3}\pattern{scale = 0.4}{3}{1/3
,2/2,3/1}{0/1,0/2,1/2,2/2,3/2,3/0,3/1,3/2,3/3}$; &\hspace{-0.3cm}
A23 = $\pattern{scale = 0.4}{3}{1/1,2/2,3/3}{0/3,1/3,2/3,3/3,1/0,2/0,2/1,2/2}\pattern{scale = 0.4}{3}{1/3,2/2,3/1}{0/3,1/3,2/3,3/3,1/0,2/0,2/1,2/2}$; &\hspace{-0.3cm}
A24 = $\pattern{scale = 0.4}{3}{1/1,2/2,3/3}{0/3,1/3,2/3,3/3,1/0,2/0,1/1,1/2}\pattern{scale = 0.4}{3}{1/3
,2/2,3/1}{0/3,1/3,2/3,3/3,1/0,2/0,1/1,1/2}$; \\[5pt]
A25 =  $\pattern{scale = 0.4}{3}{1/1,2/2,3/3}{0/0,0/1,0/2,1/1,1/2,3/0,3/1,3/2}\pattern{scale = 0.4}{3}{1/3,2/2,3/1}{0/0,0/1,0/2,1/1,1/2,3/0,3/1,3/2}$; &\hspace{-0.3cm}
A26 = $\pattern{scale = 0.4}{3}{1/1,2/2,3/3}{1/0,2/0,3/0,1/1,2/1,1/3,2/3,3/3}\pattern{scale = 0.4}{3}{1/3
,2/2,3/1}{1/0,2/0,3/0,1/1,2/1,1/3,2/3,3/3}$;  &\hspace{-0.3cm}
A27 =  $\pattern{scale = 0.4}{3}{1/1,2/2,3/3}{0/1,0/2,0/3,2/1,2/2,3/1,3/2,3/3}\pattern{scale = 0.4}{3}{1/3,2/2,3/1}{0/1,0/2,0/3,2/1,2/2,3/1,3/2,3/3}$; & \hspace{-0.3cm}
A28 = $\pattern{scale =
0.4}{3}{1/1,2/2,3/3}{0/1,0/2,0/3,1/1,1/2,3/1,3/2,3/3}\pattern{scale = 0.4}{3}{1/3,2/2,3/1}{0/1,0/2,0/3,1/1,1/2,3/1,3/2,3/3}$; \\[5pt]
A29 = $\pattern{scale = 0.4}{3}{1/1,2/2,3/3}{0/0,1/0,2/0,1/2,2/2,0/3,1/3,2/3}\pattern{scale = 0.4}{3}{1/3,2/2,3/1}{0/0,1/0,2/0,1/2,2/2,0/3,1/3,2/3}$; &\hspace{-0.3cm}
A30 = $\pattern{scale = 0.4}{3}{1/1,2/2,3/3}{0/0,1/0,2/0,1/1,2/1,0/3,1/3,2/3}\pattern{scale = 0.4}{3}{1/3,2/2,3/1}{0/0,1/0,2/0,1/1,2/1,0/3,1/3,2/3}$; &\hspace{-0.3cm}
A31 = $\pattern{scale = 0.4}{3}{1/1,2/2,3/3}{1/0,2/0,3/0,1/2,2/2,1/3,2/3,3/3}\pattern{scale = 0.4}{3}{1/3,2/2,3/1}{1/0,2/0,3/0,1/2,2/2,1/3,2/3,3/3}$; & \hspace{-0.3cm}
A32 = $\pattern{scale = 0.4}{3}{1/1,2/2,3/3}{0/0,0/1,0/2,2/1,2/2,3/0,3/1,3/2}\pattern{scale = 0.4}{3}{1/3,2/2,3/1}{0/0,0/1,0/2,2/1,2/2,3/0,3/1,3/2}$; \\[5pt]
A33 = $\pattern{scale = 0.4}{3}{1/1,2/2,3/3}{0/2,1/0,1/1,1/2,2/2,3/0,3/1,3/2}\pattern{scale = 0.4}{3}{1/3,2/2,3/1}{0/2,1/0,1/1,1/2,2/2,3/0,3/1,3/2}$; &\hspace{-0.3cm}
A34 = $\pattern{scale = 0.4}{3}{1/1,2/2,3/3}{0/1,0/2,0/3,1/1,2/1,2/2,2/3,3/1}\pattern{scale = 0.4}{3}{1/3,2/2,3/1}{0/1,0/2,0/3,1/1,2/1,2/2,2/3,3/1}$;  &\hspace{-0.3cm}
A35 = $\pattern{scale = 0.4}{3}{1/1,2/2,3/3}{0/1,0/3,1/1,1/3,2/0,2/1,2/1,2/2,2/3}\pattern{scale = 0.4}{3}{1/3,2/2,3/1}{0/1,0/3,1/1,1/3,2/0,2/1,2/1,2/2,2/3}$; & \hspace{-0.3cm}
A36 = $\pattern{scale = 0.4}{3}{1/1,2/2,3/3}{1/0,1/1,1/2,1/3,2/0,2/2,3/0,3/2}\pattern{scale = 0.4}{3}{1/3,2/2,3/1}{1/0,1/1,1/2,1/3,2/0,2/2,3/0,3/2}$. 
\end{tabular}
\end{center}

Recall that  the {\em unsigned Stirling number of the first kind} (\cite[A132393]{OEIS}) is defined by $c(n,k)=0$ if $n<k$ or $k=0$, except $c(0,0)=1$, and 
\[
c(n,k)=(n-1)c(n-1,k)+c(n-1,k-1).
\]
\begin{thm}\label{thm-pairs-39-47}
   We have $\hspace{-1mm}\pattern{scale = 0.5}{3}{1/1,2/2,3/3}{0/0,0/1,0/2,0/3,1/1,2/1,3/1,3/2}\sim_{jd}\hspace{-0.1cm}\pattern{scale =0.5}{3}{1/3,2/2,3/1}{0/0,0/1,0/2,0/3,1/1,2/1,3/1,3/2}$, $\hspace{-0.1cm}\pattern{scale = 0.5}{3}{1/1,2/2,3/3}{0/0,0/1,0/2,0/3,1/2,2/2,3/2,3/1}\sim_{jd}\hspace{-0.1cm}\pattern{scale = 0.5}{3}{1/3,2/2,3/1}{0/0,0/1,0/2,0/3,1/2,2/2,3/2,3/1} $,  $\pattern{scale = 0.5}{3}{1/1,2/2,3/3}{0/0,1/0,2/0,3/0,1/1,1/2,1/3,2/3}\sim_{jd}
   \hspace{-0.1cm}\pattern{scale = 0.5}{3}{1/3,2/2,3/1}{0/0,1/0,2/0,3/0,1/1,1/2,1/3,2/3}$, $\pattern{scale = 0.5}{3}{1/1,2/2,3/3}{0/0,1/0,2/0,3/0,2/1,2/2,2/3,1/3}\sim_{jd}
   \hspace{-0.1cm}\pattern{scale = 0.5}{3}{1/3,2/2,3/1}{0/0,1/0,2/0,3/0,2/1,2/2,2/3,1/3}$, $\pattern{scale = 0.5}{3}{1/1,2/2,3/3}{0/1,0/2,1/1,2/1,3/1,3/0,3/2,3/3}\sim_{jd}
   \hspace{-0.1cm}\pattern{scale = 0.5}{3}{1/3,2/2,3/1}{0/1,0/2,1/1,2/1,3/1,3/0,3/2,3/3}$, $\pattern{scale = 0.5}{3}{1/1,2/2,3/3}{0/1,0/2,1/2,2/2,3/2,3/1,3/0,3/2,3/3}\sim_{jd}\hspace{-0.1cm}\pattern{scale = 0.5}{3}{1/3
,2/2,3/1}{0/1,0/2,1/2,2/2,3/2,3/0,3/1,3/2,3/3} $, $\pattern{scale = 0.5}{3}{1/1,2/2,3/3}{0/3,1/3,2/3,3/3,1/0,2/0,2/1,2/2}\sim_{jd}\hspace{-0.1cm}\pattern{scale = 0.5}{3}{1/3,2/2,3/1}{0/3,1/3,2/3,3/3,1/0,2/0,2/1,2/2}$ and $\pattern{scale = 0.5}{3}{1/1,2/2,3/3}{0/3,1/3,2/3,3/3,1/0,2/0,1/1,1/2}\sim_{jd}\hspace{-0.1cm}\pattern{scale = 0.5}{3}{1/3
,2/2,3/1}{0/3,1/3,2/3,3/3,1/0,2/0,1/1,1/2}$ (Nr.~A17 to A24). Moreover, for $n\geq 2$, all the pairs satisfy 
\begin{align} \label{TK-pair-39}
T_{n,k,\ell}=\left\{  
\begin{array}{llll}  
\binom{k+\ell+2}{k+1} c(n-1,k+\ell+2), & \text{} k \geq 1, \ell \geq 1 \\[6pt]
(k+2) c(n-1,k+2) + c(n-1,k+1), & \text{} k\geq 1, \ell=0 , \\[6pt]
(\ell+2) c(n-1,\ell+2) + c(n-1,\ell+1), & \text{} \ell\geq 1, k=0 , \\[6pt]
2(n-2)!\,(1+\sum_{i=1}^{n-2}\frac{1}{i}), & \text{} k=\ell=0.
\end{array}  
\right.
\end{align}
Equivalently,
\begin{align}\label{T-pair-39-generating}
  T_n(x,y)=&\sum_{i=0}^{n-2} c(n-1,i+1)(x^i+y^i)  \notag \\
& \   + \sum_{k=0}^{n-2} \sum_{\ell=0}^{n-2}\binom{k+\ell+2}{k+1} c(n-1,k+\ell+2) \, x^k y^\ell.  
\end{align}
\end{thm}
Instead of directly finding the recurrence relations of any pair of patterns, we first consider two equidistributed patterns of length two by removing one row and one column that contains the first element of Nr.~A17. 
Let $\widetilde{T}_{n,k}$ be the number of $n$-permutations with $k$ occurrences of $\pattern{scale = 0.5}{2}{1/1,2/2}{0/0,1/0,2/0,2/1}$.

\begin{lem}\label{JL2} 
We have $\pattern{scale = 0.5}{2}{1/1,2/2}{0/0,1/0,2/0,2/1}\sim_d\hspace{-0.1cm}\pattern{scale =0.5}{2}{1/2,2/1}{0/1,1/1,2/1,2/2}$. Moreover, the number $\widetilde{T}_{n,k}$ coincides with the unsigned Stirling number of the first kind $c(n,k+1)$. 
\end{lem}
\begin{proof}
By the complement operation, we have 
$
\pattern{scale =0.5}{2}{1/2,2/1}{0/1,1/1,2/1,2/2}\xrightarrow{c}\pattern{scale =0.5}{2}{1/1,2/2}{0/1,1/1,2/1,2/0}.
$
Thus, it is equivalent to prove $\pattern{scale = 0.5}{2}{1/1,2/2}{0/0,1/0,2/0,2/1}\sim_d\hspace{-0.1cm}\pattern{scale =0.5}{2}{1/1,2/2}{0/1,1/1,2/1,2/0}$. First,
to prove
\begin{align}\label{T=stirling}
\widetilde{T}_{n,k}=c(n,k+1),
\end{align}
we shall show
\begin{align}\label{TJ}
   \widetilde{T}_{n,k}=(n-1)\widetilde{T}_{n-1,k}+\widetilde{T}_{n-1,k-1},
\end{align}
together with  $\widetilde{T}_{1,0}=c(1,1)=1$
holds for both patterns. We obtain each permutation counted by the left side of~(\ref{TJ}) by inserting the largest element $n$ into an $(n-1)$-permutation.

\begin{itemize}
\item
For the pattern $\pattern{scale = 0.5}{2}{1/1,2/2}{0/0,1/0,2/0,2/1}$, the insertion of $n$ in the last slot creates an additional occurrence of the pattern with $1n$, which contributes to
$\widetilde{T}_{n-1,k-1}$. In any other slot, the insertion  preserves the number of occurrences, which is counted by
$(n-1)\widetilde{T}_{n-1,k}$.

\item
For the pattern $\pattern{scale=0.5}{2}{1/1,2/2}{0/1,1/1,2/1,2/0}$, the descriptions for the terms in~(\ref{TJ}) are the same as those for $\pattern{scale = 0.5}{2}{1/1,2/2}{0/0,1/0,2/0,2/1}$, except for the extra occurrence of $\pattern{scale=0.5}{2}{1/1,2/2}{0/1,1/1,2/1,2/0}$ when considering $\widetilde{T}_{n-1,k-1}$, which corresponds to $(n-1)n$. 
\end{itemize}
This completes the proof.
\end{proof}

\begin{proof}[Proof of Theorem~\ref{thm-pairs-39-47}]

Applying the complement, reverse, and inverse operations to Nr.~A17, we have
\[
\pattern{scale = 0.5}{3}{1/1,2/2,3/3}{0/0,0/1,0/2,0/3,1/1,2/1,3/1,3/2}\xrightarrow{c}\pattern{scale = 0.5}{3}{1/3,2/2,3/1}{0/0,0/1,0/2,0/3,1/2,2/2,3/2,3/1}\xrightarrow{r}\pattern{scale = 0.5}{3}{1/1,2/2,3/3}{0/1,0/2,1/2,2/2,3/2,3/1,3/0,3/2,3/3}\xrightarrow{c}\pattern{scale = 0.5}{3}{1/3,2/2,3/1}{0/1,0/2,1/1,2/1,3/1,3/0,3/2,3/3},\]
\[
\pattern{scale = 0.5}{3}{1/1,2/2,3/3}{0/0,0/1,0/2,0/3,1/1,2/1,3/1,3/2}\xrightarrow{i}\pattern{scale = 0.5}{3}{1/1,2/2,3/3}{0/0,1/0,2/0,3/0,1/1,1/2,1/3,2/3}\xrightarrow{r}\pattern{scale = 0.5}{3}{1/3,2/2,3/1}{0/0,1/0,2/0,3/0,2/1,2/2,2/3,1/3}\xrightarrow{c}\pattern{scale = 0.5}{3}{1/1,2/2,3/3}{0/3,1/3,2/3,3/3,1/0,2/0,2/1,2/2}\xrightarrow{r}\pattern{scale = 0.5}{3}{1/3,2/2,3/1}{0/3,1/3,2/3,3/3,1/0,2/0,1/1,1/2}.
\]
The operations for the pattern $\pattern{scale=0.5}{3}{1/3,2/2,3/1}{0/0,0/1,0/2,0/3,1/1,2/1,3/1,3/2}$ are similar to those for $\pattern{scale=0.5}{3}{1/1,2/2,3/3}{0/0,0/1,0/2,0/3,1/1,2/1,3/1,3/2}$ for pairs Nr.~A17 to Nr.~A24. Thus, it suffices to prove the joint equidistribution of Nr.~A17. Here, we let
\[q_1=\pattern{scale = 0.5}{3}{1/1,2/2,3/3}{0/0,0/1,0/2,0/3,1/1,2/1,3/1,3/2}, \hspace{0.6cm}q_2=\pattern{scale = 0.5}{3}{1/3,2/2,3/1}{0/0,0/1,0/2,0/3,1/1,2/1,3/1,3/2}.
\]
The first element in each occurrence of the two patterns in a permutation~$\pi=\pi_1\pi_2\cdots \pi_n$ must be the first element $\pi_1$. We define $A=\{\pi_j:\pi_j>\pi_1\}$ and $B=\{\pi_j:\pi_j<\pi_1\}$.
For each permutation $\pi$ with $k$ occurrences of $q_1$ and $\ell$ occurrences of $q_2$, the permutation formed by $A$ contains $k$ occurrences of $\pattern{scale = 0.5}{2}{1/1,2/2}{0/0,1/0,2/0,2/1}$, and the permutation formed by $B$ contains $\ell$ occurrences of $\pattern{scale =0.5}{2}{1/2,2/1}{0/1,1/1,2/1,2/2}$. 
By Lemma~\ref{JL2}, we have
\begin{align}\label{T_n,k,l-pair39-2-length}
T_{n,k,\ell}=\sum_{i=0}^{n-1} \binom{n-1}{i}\, \widetilde{T}_{i,k}\,\widetilde{T}_{n-i-1,\ell}.  
\end{align}
It follows that $T_{n,k,\ell}= T_{n,\ell,k}$. This proves the joint equidistribution of $q_1$ and $q_2$.

 Moreover, \eqref{T-pair-39-generating} is obtained from \eqref{T_n,k,l-pair39-2-length} as follows:
 \begin{small}
\begin{align*}
T_n(x,y)=&\sum_{k=0}^{n-2}\sum_{\ell=0}^{n-2} \sum_{i=0}^{n-1} \binom{n-1}{i}\, \widetilde{T}_{i,k}\,\widetilde{T}_{n-i-1,\ell}x^ky^\ell\\
=&\sum_{\ell=0}^{n-2} \widetilde{T}_{0,0}\,c(n-1,\ell+1)y^\ell+ \sum_{k=0}^{n-2} c(n-1,k+1) \widetilde{T}_{0,0}\,x^k   \\
&+\sum_{k=0}^{n-2}\,\sum_{\ell=0}^{n-2}\, \sum_{i=1}^{n-2} \binom{n-1}{i}\, c(i,k+1)c(n-i-1,\ell+1)x^ky^\ell\\
=& \sum_{i=0}^{n-2} c(n-1,i+1)(x^i+y^i) \, + \sum_{k=0}^{n-2}\,\sum_{\ell=0}^{n-2}\, \sum_{i=k+1}^{n-2-\ell} \binom{n-1}{i}c(i,k+1)\,c(n-i-1,\ell+1)x^ky^\ell\\
=&\sum_{i=0}^{n-2} c(n-1,i+1)(x^i+y^i) \, + \sum_{k=0}^{n-2} \sum_{\ell=0}^{n-2}\binom{k+\ell+2}{k+1} c(n-1,k+\ell+2) \, x^k y^\ell.
\end{align*}
\end{small}
\noindent
The three terms in the second equality correspond to the cases $i=0$, $i=n-1$, and $i= 1,\ldots,n-2$, respectively. The last equality follows from a variation of the Chu-Vandermonde identity $$\sum_{i=m-r}^{n-r}c(i,m-r)c(n-i,r) = \binom{m}{r} c(n,m),$$ see \cite{AM-S} or \cite[Equation (17)]{mathworld}. 
For \eqref{TK-pair-39}, by taking the coefficients of both sides of \eqref{T-pair-39-generating}, we can directly obtain the number of $T_{n,k,\ell}$ for $k+\ell \neq 0$. In particular, for $k=\ell=0$, we have 
     \begin{align}\label{eq-har}
        T_{n,0,0}=& \binom{2}{1}\,c(n-1,2)+2c(n-1,1) \notag \\
        =& 2(n-2)!\sum_{i=1}^{n-2}\frac{1}{i}+2(n-2)! \notag \\
        =& 2(n-2)!\,(1+\sum_{i=1}^{n-2}\frac{1}{i}),
    \end{align}
where $c(n,2)=(n-1)!\sum_{i=1}^{n-1}\frac{1}{i}$, see~\cite[Exercise 1.43]{Stanley2012}.
Therefore, we complete the proof.
\end{proof}

\begin{remark}
The number of $n$-permutations that avoid both $q_1$ and $q_2$ is  $T_{n,0,0}=2a_{n-2}$, where $a_n =n!\,(1+\sum_{i=1}^{n}\frac{1}{i})$ is the $n$-harmonic number that appears in OEIS \cite[A000774]{OEIS}, see also~\cite{KR}.
\end{remark}

In the subsequent two theorems, we establish the recurrence relation by considering three cases concerning the position of the largest element. Subsequently,  we employ the notation
\begin{align*}
 T^{(1)}_{n,k,\ell}:&=\# \{\pi \in \T(n,k,\ell): \pi_1=n \},\\
T^{(2)}_{n,k,\ell}:&=\# \{\pi \in \T(n,k,\ell): \pi_n=n \},\\
T^{(3)}_{n,k,\ell}:&=\# \{\pi \in \T(n,k,\ell): \pi_i= n , i\neq 1, n\}.   
\end{align*}
It is clear that
\begin{align}\label{T_n,k,l}
T_{n,k,\ell}=T^{(1)}_{n,k,\ell}+T^{(2)}_{n,k,\ell}+T^{(3)}_{n,k,\ell}.
\end{align}
\begin{thm}\label{thm-pairs-A25-32}
We have $\pattern{scale = 0.5}{3}{1/1,2/2,3/3}{0/0,0/1,0/2,1/1,1/2,3/0,3/1,3/2}\sim_{jd}\hspace{-0.1cm}\pattern{scale =0.5}{3}{1/3,2/2,3/1}{0/0,0/1,0/2,1/1,1/2,3/0,3/1,3/2}$, $\pattern{scale = 0.5}{3}{1/1,2/2,3/3}{0/0,1/0,2/0,1/1,2/1,0/3,1/3,2/3}\sim_{jd}\hspace{-0.1cm}\pattern{scale = 0.5}{3}{1/3,2/2,3/1}{0/0,1/0,2/0,1/1,2/1,0/3,1/3,2/3}$, $\pattern{scale = 0.5}{3}{1/1,2/2,3/3}{0/0,0/1,0/2,2/1,2/2,3/0,3/1,3/2}\sim_{jd} \hspace{-0.1cm}\pattern{scale = 0.5}{3}{1/3,2/2,3/1}{0/0,0/1,0/2,2/1,2/2,3/0,3/1,3/2}$, $\pattern{scale = 0.5}{3}{1/1,2/2,3/3}{0/1,0/2,0/3,1/1,1/2,3/1,3/2,3/3}\sim_{jd} \hspace{-0.1cm}\pattern{scale = 0.5}{3}{1/3,2/2,3/1}{0/1,0/2,0/3,1/1,1/2,3/1,3/2,3/3}$, $\pattern{scale = 0.5}{3}{1/1,2/2,3/3}{0/1,0/2,0/3,2/1,2/2,3/1,3/2,3/3}\sim_{jd}\hspace{-0.1cm}\pattern{scale = 0.5}{3}{1/3,2/2,3/1}{0/1,0/2,0/3,2/1,2/2,3/1,3/2,3/3}$,
$\pattern{scale = 0.5}{3}{1/1,2/2,3/3}{1/0,2/0,3/0,1/1,2/1,1/3,2/3,3/3}\sim_{jd}\hspace{-0.1cm}\pattern{scale = 0.5}{3}{1/3
,2/2,3/1}{1/0,2/0,3/0,1/1,2/1,1/3,2/3,3/3}$, $\pattern{scale = 0.5}{3}{1/1,2/2,3/3}{1/0,2/0,3/0,1/2,2/2,1/3,2/3,3/3}\sim_{jd}\hspace{-0.1cm}\pattern{scale = 0.5}{3}{1/3,2/2,3/1}{1/0,2/0,3/0,1/2,2/2,1/3,2/3,3/3}$ and $\pattern{scale = 0.5}{3}{1/1,2/2,3/3}{0/0,1/0,2/0,1/2,2/2,0/3,1/3,2/3}\sim_{jd} \hspace{-0.1cm}\pattern{scale = 0.5}{3}{1/3,2/2,3/1}{0/0,1/0,2/0,1/2,2/2,0/3,1/3,2/3}$ (Nr.~A25 to A32) . 
\end{thm}
\begin{proof}
Applying the complement, reverse, and inverse operations to Nr.~A25, we have
\[
 \pattern{scale = 0.5}{3}{1/1,2/2,3/3}{0/0,0/1,0/2,1/1,1/2,3/0,3/1,3/2}\xrightarrow{r}\pattern{scale = 0.5}{3}{1/3,2/2,3/1}{0/0,0/1,0/2,2/1,2/2,3/0,3/1,3/2}\xrightarrow{c}\pattern{scale = 0.5}{3}{1/1,2/2,3/3}{0/1,0/2,0/3,2/1,2/2,3/1,3/2,3/3}\xrightarrow{r}\pattern{scale = 0.5}{3}{1/3,2/2,3/1}{0/1,0/2,0/3,1/1,1/2,3/1,3/2,3/3},
 \]
 \[
 \pattern{scale = 0.5}{3}{1/1,2/2,3/3}{0/0,0/1,0/2,1/1,1/2,3/0,3/1,3/2}\xrightarrow{i}\pattern{scale = 0.5}{3}{1/1,2/2,3/3}{0/0,1/0,2/0,1/1,2/1,0/3,1/3,2/3}\xrightarrow{c}\pattern{scale = 0.5}{3}{1/3,2/2,3/1}{0/0,1/0,2/0,1/2,2/2,0/3,1/3,2/3}\xrightarrow{r}\pattern{scale = 0.5}{3}{1/1,2/2,3/3}{1/0,2/0,3/0,1/2,2/2,1/3,2/3,3/3}\xrightarrow{c}\pattern{scale = 0.5}{3}{1/3,2/2,3/1}{1/0,2/0,3/0,1/1,2/1,1/3,2/3,3/3}.
\]
The procedures for the pattern $\pattern{scale = 0.5}{3}{1/3,2/2,3/1}{0/0,0/1,0/2,1/1,1/2,3/0,3/1,3/2}$ are similar to those for $\pattern{scale = 0.5}{3}{1/1,2/2,3/3}{0/0,0/1,0/2,1/1,1/2,3/0,3/1,3/2}$. Thus, it suffices to prove the joint equidistribution of Nr.~A25.
 Here, we let
\[
q_1= \pattern{scale = 0.5}{3}{1/1,2/2,3/3}{0/0,0/1,0/2,1/1,1/2,3/0,3/1,3/2} 
\hspace{0.2cm}\mbox{and} \hspace{0.2cm}
 q_2= \pattern{scale = 0.5}{3}{1/3,2/2,3/1}{0/0,0/1,0/2,1/1,1/2,3/0,3/1,3/2}.
 \]
To prove $T_{n,k,\ell}=T_{n,\ell,k}$, we shall show
\[
T^{(1)}_{n,k,\ell}=T^{(2)}_{n,\ell,k},\hspace{0.5cm} T^{(2)}_{n,k,\ell}=T^{(1)}_{n,\ell,k}, \hspace{0.5cm} T^{(3)}_{n,k,\ell}=T^{(3)}_{n,\ell,k}.
\]
We claim that
\begin{align}
     T^{(1)}_{n,k,\ell}&= T^{(1)}_{n-1,k,\ell-1} + T^{(2)}_{n-1,k,\ell}+T^{(3)}_{n-1,k,\ell-1}, \label{T1,n,k,l}\\
     T^{(2)}_{n,k,\ell}&= T^{(1)}_{n-1,k,\ell} + T^{(2)}_{n-1,k-1,\ell}+T^{(3)}_{n-1,k-1,\ell}, \label{T2,n,k,l}\\
     T^{(3)}_{n,k,\ell}&=(n-2) T_{n-1,k,\ell}. \label{T3,n,k,l}
 \end{align}    
The above three equations can be proved by considering the ways of getting  a permutation $\pi=\pi_1\pi_2\cdots \pi_n$ from some $\pi' =\pi'_1\pi'_2\cdots \pi'_{n-1}\in S_{n-1}$ by inserting the largest element $n$ into a slot.
\begin{itemize}
    \item For \eqref{T1,n,k,l}, in each case, $n$ is inserted into the first slot. The situation where $\pi'_1=n-1$ or $\pi'_i=n-1$, an extra occurrence of $q_2$ would be created, in the form of $n\pi_j\pi_n$, where $\pi_j$ is the leftmost element larger than $\pi_n$ and distinct from $n$. Thus, the two cases give $T^{(1)}_{n-1,k,\ell-1}$ and $T^{(3)}_{n-1,k,\ell-1}$, respectively. The situation where $\pi'_{n-1}=n-1$ preserves the number of both patterns, which contributes to $T^{(2)}_{n-1,k,\ell}$. 
\item For \eqref{T2,n,k,l}, $n$ is inserted into the last slot in each case. First, when $\pi'_1=n-1$, the insertion preserves the number of occurrences of both patterns, which gives $T^{(1)}_{n-1,k,\ell}$. Other situations  result in  an extra occurrence of $q_1$, which is formed by $\pi_1\pi_jn$, where $\pi_j$ is the leftmost element larger than $\pi_1$. Thus, these cases contribute to $T^{(2)}_{n-1,k-1,\ell}+T^{(3)}_{n-1,k-1,\ell}$. Summing over the above cases we get \eqref{T2,n,k,l}.
\item For \eqref{T3,n,k,l}, the permutations counted by the left-hand side correspond to the insertion of $n$ in any slot other than the first and last. In each case, the number of occurrences of both patterns remains unchanged. Thus, we get the right-hand side of \eqref{T3,n,k,l}.
\end{itemize} 

The initial conditions are given by $T^{(1)}_{2,0,0}=T^{(2)}_{2,0,0}=1$, $T^{(1)}_{3,0,1}=T^{(2)}_{3,1,0}=1$, $T^{(1)}_{3,1,0}=T^{(2)}_{3,0,1}=0$, $T^{(3)}_{3,0,1}=T^{(3)}_{3,1,0}=0$, and $T^{(3)}_{3,1,0}=T^{(3)}_{3,0,1}=0$.  
Thus, by using induction on $n$, $k$ and $\ell$, we complete the proof.
\end{proof}

\begin{remark} 
The number $T_{n,0,0}$ in Theorem~\ref{thm-pairs-A25-32} appears in OEIS~\cite[A054091]{OEIS}. 
\end{remark}

\begin{thm} \label{thm-pair-A33-36}
We have $\pattern{scale = 0.5}{3}{1/1,2/2,3/3}{0/2,1/0,1/1,1/2,2/2,3/0,3/1,3/2}\sim_{jd}\hspace{-0.1cm}\pattern{scale = 0.5}{3}{1/3,2/2,3/1}{0/2,1/0,1/1,1/2,2/2,3/0,3/1,3/2}$, $\pattern{scale = 0.5}{3}{1/1,2/2,3/3}{0/1,0/2,0/3,1/1,2/1,2/2,2/3,3/1}\sim_{jd} \hspace{-0.1cm}\pattern{scale = 0.5}{3}{1/3,2/2,3/1}{0/1,0/2,0/3,1/1,2/1,2/2,2/3,3/1}$, $\pattern{scale = 0.5}{3}{1/1,2/2,3/3}{0/1,0/3,1/1,1/3,2/0,2/1,2/1,2/2,2/3}\sim_{jd} \hspace{-0.1cm}\pattern{scale = 0.5}{3}{1/3,2/2,3/1}{0/1,0/3,1/1,1/3,2/0,2/1,2/1,2/2,2/3}$
 and $\pattern{scale = 0.5}{3}{1/1,2/2,3/3}{1/0,1/1,1/2,1/3,2/0,2/2,3/0,3/2}\sim_{jd} \hspace{-0.1cm}\pattern{scale = 0.5}{3}{1/3,2/2,3/1}{1/0,1/1,1/2,1/3,2/0,2/2,3/0,3/2}$ (Nr.~A33 to A36). Moreover, all the pairs satisfy the recurrence relation
\begin{align}
T_{n,k,\ell}
= (n-2)T_{n-1,k,\ell}+T_{n-1,k-1,\ell}+T_{n-1,k,\ell-1}  +T_{n-2,k,\ell}-T_{n-2,k-1,\ell-1},\hspace{0.3cm} n \geq 4,
\label{T-A33-recurr}
\end{align}
with the initial conditions $T_{2,0,0}=2$, $T_{3,0,0}=4$, $T_{3,1,0}=T_{3,0,1}=1$,
and $T_{n,k,\ell}=0$ for $k, \ell \geq n-1$.
 Equivalently,
 \begin{align} \label{T-A33-Genera}
 T_n(x,y)= (n+x+y-2) T_{n-1}(x,y)+(1-xy) T_{n-2}(x,y),\hspace{0.3cm} n\geq 4,
 \end{align}
 with the initial conditions $T_2(x,y)=2$ and $T_3(x,y)=x+y+4$.
\end{thm}
\begin{proof}
Applying the complement, reverse, and inverse operations to Nr.~A33, we have
\[
 \pattern{scale = 0.5}{3}{1/1,2/2,3/3}{0/2,1/0,1/1,1/2,2/2,3/0,3/1,3/2}\xrightarrow{cr}\pattern{scale = 0.5}{3}{1/1,2/2,3/3}{0/1,0/2,0/3,1/1,2/1,2/2,2/3,3/1},
 \]
 \[
 \pattern{scale = 0.5}{3}{1/1,2/2,3/3}{0/2,1/0,1/1,1/2,2/2,3/0,3/1,3/2}\xrightarrow{i}\pattern{scale = 0.5}{3}{1/1,2/2,3/3}{0/1,0/3,1/1,1/3,2/0,2/1,2/1,2/2,2/3}\xrightarrow{cr}\pattern{scale = 0.5}{3}{1/1,2/2,3/3}{1/0,1/1,1/2,1/3,2/0,2/2,3/0,3/2}.
\]
 The procedures for the pattern $\pattern{scale = 0.5}{3}{1/3,2/2,3/1}{0/2,1/0,1/1,1/2,2/2,3/0,3/1,3/2}$ are similar to those for $\pattern{scale = 0.5}{3}{1/1,2/2,3/3}{0/2,1/0,1/1,1/2,2/2,3/0,3/1,3/2}$. Thus, it suffices to prove the joint equidistribution of Nr.~A33. Here, we let
 \[
 q_1=\pattern{scale = 0.5}{3}{1/1,2/2,3/3}{0/2,1/0,1/1,1/2,2/2,3/0,3/1,3/2}
\hspace{0.2cm}\mbox{and}\hspace{0.2cm}
q_2=\pattern{scale = 0.5}{3}{1/3,2/2,3/1}{0/2,1/0,1/1,1/2,2/2,3/0,3/1,3/2}.
\] 
The initial conditions follow directly from the definition. We shall prove that for $n\geq 4$
\begin{align}
   T^{(1)}_{n,k,\ell}&= T^{(1)}_{n-1,k,\ell-1} + T^{(2)}_{n-1,k,\ell}+T^{(3)}_{n-1,k,\ell}, \label{H1,n,k,l}\\
     T^{(2)}_{n,k,\ell}&= T^{(1)}_{n-1,k,\ell} + T^{(2)}_{n-1,k-1,\ell}+T^{(3)}_{n-1,k-1,\ell}, \label{H2,n,k,l}\\
     T^{(3)}_{n,k,\ell}&=(n-2) T_{n-1,k,\ell} +T^{(3)}_{n-1,k,\ell-1}-T^{(3)}_{n-1,k,\ell}. \label{H3,n,k,l}
\end{align}
For each $n$-permutation $\pi=\pi_1\pi_2\cdots \pi_n$ counted by the left-hand side of each equation, we obtain it by inserting the largest element $n$ into an  $(n-1)$-permutation $\pi'=\pi'_1\pi'_2\cdots \pi'_{n-1}$.
\begin{itemize}
    \item For \eqref{H1,n,k,l}, in each case, $n$ is inserted into the first slot. In the situation where $\pi'_1=n-1$, an extra occurrence of $q_2$ would be created, in the form of $n(n-1)\pi_n$. Thus, this case gives $T^{(1)}_{n-1,k,\ell-1}$. The two situations where $\pi'_{n-1}=n-1$ or $\pi'_i=n-1$ preserve the number of both patterns contributing to $T^{(2)}_{n-1,k,\ell}$ and $T^{(3)}_{n-1,k,\ell}$, respectively. 
\item For \eqref{H2,n,k,l}, $n$ is inserted into the last slot in each case. First, when $\pi'_1=n-1$, the insertion maintains the number of occurrences of both patterns, which gives $T^{(1)}_{n-1,k,\ell}$. Other situations would result in  an extra occurrence of $q_1$, which is formed by $\pi_{j}\pi_{j+1}n$, where $\pi_{j+1}=n-1$. Thus, these cases together contribute to $T^{(2)}_{n-1,k-1,\ell}+T^{(3)}_{n-1,k-1,\ell}$. Summing over the above cases we get the equation.
\item For \eqref{H3,n,k,l}, the permutations counted by the left-hand side correspond to the insertion of $n$ in any slot other than the first and last. The situations where $\pi'_1=n-1$ and $\pi'_{n-1}=n-1$ maintain the number of occurrences of both patterns. Thus, these cases contribute to $(n-2)T^{(1)}_{n-1,k,\ell}$ and $(n-2)T^{(2)}_{n-1,k,\ell}$, respectively. In the case where $\pi'_i=n-1$, inserting $n$  immediately to the left of $\pi'_i$ leads to an occurrence of $q_2$, formed by $n(n-1)\pi_n$  contributing to $T^{(3)}_{n-1,k,\ell-1}$. In any other slot,  the number of both patterns remains unchanged, resulting in $(n-3)T^{(3)}_{n-1,k,\ell}$. Summing over the above cases and using~\eqref{T_n,k,l}, we obtain the right side.
\end{itemize}                              
Let $H_{n,k,\ell}=T^{(1)}_{n,k-1,\ell}+ T^{(2)}_{n,k,\ell-1}+T^{(3)}_{n,k,\ell}$.
By using \eqref{H1,n,k,l}, \eqref{H2,n,k,l}, and \eqref{H3,n,k,l}, we have
\begin{align*}
     T_{n,k,\ell}=  &\,(n-1)T_{n-1,k,\ell}+ T_{n-1,k-1,\ell}+T_{n-1,k,\ell-1}-H_{n-1,k,\ell},
\end{align*}
and
\begin{align*}
    H_{n,k,\ell}
    &= \, (n-2)T_{n-1,k,\ell}+ T_{n-1,k-1,\ell-1}+T_{n-1,k-1,\ell}+T_{n-1,k,\ell-1} -H_{n-1,k,\ell}\\[6pt]
    & = \, T_{n,k,\ell}-T_{n-1,k,\ell}+ T_{n-1,k-1,\ell-1}.
\end{align*}
Combining the two equations, we obtain \eqref{T-A33-recurr}. Thus, \eqref{T-A33-Genera} is derived from \eqref{T-A33-recurr} as follows:
\begin{small}
\begin{align*}
 T_n(x,y)=
 &\sum_{k=0}^{n-2} \sum_{\ell=0}^{n-2} 
 \left((n-2) T_{n-1,k,\ell}+ T_{n-1,k-1,\ell}+ T_{n-1,k,\ell-1}\right)x^ky^\ell +\sum_{k=0}^{n-2} \sum_{\ell=0}^{n-2}\left(T_{n-2,k,\ell}-T_{n-2,k-1,\ell-1}\right)x^ky^\ell\\
 =& (n-2)\sum_{k=0}^{n-3} \sum_{\ell=0}^{n-3}
 T_{n-1,k,\ell} x^k y^\ell+ x \sum_{k=0}^{n-2} \sum_{\ell=0}^{n-3}T_{n-1,k-1,\ell} x^{k-1} y^\ell+ y \sum_{k=0}^{n-3} \sum_{\ell=0}^{n-2}T_{n-1,k,\ell-1} x^k y^{\ell-1} \\
 &+\sum_{k=0}^{n-4} \sum_{\ell=0}^{n-4}T_{n-2,k,\ell} x^k y^{\ell} - x y \sum_{k=0}^{n-3} \sum_{\ell=0}^{n-3}T_{n-2,k-1,\ell-1} x^{k-1} y^{\ell-1} \\
 =& \left(n+x+y-2\right) T_{n-1}(x,y)+(1-xy) T_{n-2}(x,y).
\end{align*}
\end{small}
Hence, we complete the proof. 
\end{proof}
It should be noted that all the patterns presented in Theorems \ref{thm-pairs-A25-32} and \ref{thm-pair-A33-36} exhibit the same equidistribution. Let $T_{n,k}=\sum_{\ell=0}^{n-2}T_{n,k,\ell}$ and $T_n(x)=\sum_{k=0}^{n-2}T_{n,k}x^k$.
\begin{coro}\label{coro-equi}
All the patterns in Theorem~\ref{thm-pairs-A25-32} and~\ref{thm-pair-A33-36} have the same equidistribution, and their enumeration satisfies the recurrence relation
\begin{align}
     T_{n,k} & = (n-1)T_{n-1,k}+T_{n-1,k-1}+T_{n-2,k}-T_{n-2,k-1},\hspace{0.3cm} n\geq 4, \label{TC} 
 \end{align}
 with the initial conditions $T_{2,0}=2$ and $T_{3,0}=5$.
 Equivalently,
\begin{align}\label{TCF-generating function}
	T_n(x) = (n+x-1)T_{n-1}(x)+(1-x)T_{n-2}(x),\hspace{0.3cm} n\geq 4,
	\end{align}
 with the initial conditions $T_2(x)=2$ and $T_3(x)=x+5$.	
\end{coro}
\begin{proof}
It is easy to obtain the initial conditions. For patterns in Theorem~\ref{thm-pairs-A25-32},  we have $T^{(1)}_{n,k}=T_{n-1,k}$ since in the case where $\pi_1=n$  removing $\pi_1$  does not affect any occurrence of $q_1$.
We obtain \eqref{TC} from \eqref{T1,n,k,l}, \eqref{T2,n,k,l}, and \eqref{T3,n,k,l} as follows:
\begin{align*}
 T_{n,k}= & \sum_{\ell=0}^{n-2}(T^{(1)}_{n,k,\ell}+T^{(2)}_{n,k,\ell}+
     T^{(3)}_{n,k,\ell})\\[6pt]
     =&\ T^{(1)}_{n-1,k} + T^{(2)}_{n-1,k}+T^{(3)}_{n-1,k}+T^{(1)}_{n-1,k} + T^{(2)}_{n-1,k-1}+T^{(3)}_{n-1,k-1}+(n-2) T_{n-1,k}\\[6pt]
     =& \ T_{n-1,k}+T_{n-2,k}+T_{n-1,k-1}-T_{n-2,k-1}+(n-2) T_{n-1,k}\\[6pt]
     =&\ (n-1)T_{n-1,k}+T_{n-1,k-1}+T_{n-2,k}-T_{n-2,k-1}. 
\end{align*}
 Moreover, \eqref{TCF-generating function} can be directly derived from \eqref{TC} as follows:
\begin{align*}
T_n(x)=& (n-1)\sum_{k=0}^{n-2} T_{n-1,k}x^k + \sum_{k=0}^{n-2} T_{n-1,k-1}x^k + \sum_{k=0}^{n-2} T_{n-2,k}x^k- \sum_{k=0}^{n-2} T_{n-2,k-1}x^k\\ 
=& (n-1) T_{n-1}(x)+ x\sum_{k=0}^{n-2} T_{n-1,k-1}x^{k-1} + T_{n-2}(x)- x\sum_{k=0}^{n-3} T_{n-2,k-1}x^{k-1}\\
=&(n+x-1)T_{n-1}(x) + (1-x) T_{n-2}(x).
\end{align*}

For patterns in Theorem~\ref{thm-pair-A33-36}, by setting $y=1$ in \eqref{T-A33-Genera}, we obtain \eqref{TCF-generating function}.
By taking the coefficients of both sides of \eqref{TCF-generating function}, 
we get \eqref{TC}. 
\end{proof}

\begin{remark} 
The number $T_{n,0}$ in Corollary~\ref{thm-pairs-A25-32}, which appears in OEIS~\cite[A102038]{OEIS}, coincides with  $|I_n(0\underline{11})|$, the number of inversion sequences of length $n$ avoiding the vincular pattern $0\underline{11}$.
This can be seen from \eqref{TC} and the recurrence relation of $|I_n(0\underline{11})|$ obtained by Lin and Yan~\cite[Theorem~2.3]{Lin-Yan},  see also \cite[Table~2]{JS}. 
\end{remark}

Finally, we connect $T_{n,k}$ to the enumeration of pattern-avoiding inversion sequences. 
The study of pattern avoidance in inversion sequences was initiated independently by Corteel \emph{et al.}~\cite{Corteel2016} and by Mansour and Shattuck~\cite{Mansour2015}.  
An \emph{inversion sequence of length} \( n \) is an integer sequence \( (e_1, e_2, \ldots, e_n) \) satisfying \( 0 \le e_i \le i-1 \) for all \( 1 \le i \le n \).  
Let \( \mathcal{I}(n,k) \) denote the set of inversion sequences of length \( n \) that contain exactly \( k \) occurrences of the vincular pattern \( \underline{00} \) (i.e., consecutive equal entries), and define \( I_{n,k} = |\mathcal{I}(n,k)| \).

\begin{thm}\label{thm-inversion}
 We have $I_{n,k}=T_{n,k}$ for any $n,k$.
\end{thm}
\begin{proof}
 By definition, we have $I_{n,0}=|I_n(0\underline{11})|$. The initial condition gives $I_{3,1}=T_{3,1}=1$. 
For any $e=(e_1,e_2,\ldots,e_n)$ counted by $I_{n,k}$, we distinguish three cases. 
\begin{itemize}
    \item $e_n\neq e_{n-1}$. Then, for a fixed $(e_1, e_2,\ldots, e_{n-1}) \in \mathcal{I}(n-1,k)$, there are exactly $n-1$ choices for $e_n$. Thus, this case gives $(n-1)I_{n-1,k}$. 
    
    \item $e_n= e_{n-1}=0$. In this case, there are $I_{n-2,k}$ such inversion sequences.
    
    \item $e_n= e_{n-1}\neq 0$. Then $(e_1, e_2,\ldots, e_{n-1})$ contains $k-1$ occurrences of such pattern. Combining with the count for $e_n= e_{n-1}=0$, this case contributes to $I_{n-1,k-1}-I_{n-2,k-1}$.
\end{itemize}
Thus, we have
\begin{align}
    I_{n,k}& = (n-1)I_{n-1,k}+I_{n-1,k-1}+I_{n-2,k}-I_{n-2,k-1} \label{I_n,k=T_n,k},
\end{align}
which coincides with the number $T_{n,k}$ as seen in \eqref{TC}. This completes the proof.
\end{proof}

\section*{Acknowledgements} 
The authors are grateful to Sergey Kitaev for his discussion related to this paper. 

\noindent \textbf{Funding} The work of Philip B. Zhang was supported by the National Natural Science Foundation of China (No.~12171362) and the Tianjin Municipal Natural Science Foundation (No.~ 25JCYBJC00430).

\noindent \textbf{Data Availability} There are no relevant data to share relating to the results of this publication.

\noindent \textbf{Declarations}

\noindent  \textbf{Confict of Interest} The authors have no relevant financial or non-financial interests to disclose.

\end{document}